\documentclass{article}

\usepackage[preprint]{neurips_2026}

\usepackage[utf8]{inputenc}
\usepackage[T1]{fontenc}

\usepackage{amsmath,amssymb,amsthm,mathtools}
\usepackage{bm}
\usepackage{booktabs}
\usepackage{multirow}
\usepackage{graphicx}
\usepackage{placeins}
\usepackage{enumitem}
\usepackage{algorithm}
\usepackage{algpseudocode}
\usepackage{tikz}
\usetikzlibrary{arrows.meta}
\makeatletter
\providecommand{\theHALG@line}{}
\renewcommand{\theHALG@line}{\thealgorithm.\arabic{ALG@line}}
\makeatother
\usepackage{url}
\usepackage{nicefrac}
\usepackage{microtype}
\usepackage{xcolor}
\usepackage{aliascnt}
\usepackage{hyperref}
\hypersetup{hidelinks,hypertexnames=false}
\usepackage[capitalize,noabbrev]{cleveref}

\newtheorem{theorem}{Theorem}

\newaliascnt{lemma}{theorem}
\newtheorem{lemma}[lemma]{Lemma}
\aliascntresetthe{lemma}

\newaliascnt{proposition}{theorem}
\newtheorem{proposition}[proposition]{Proposition}
\aliascntresetthe{proposition}

\newaliascnt{corollary}{theorem}
\newtheorem{corollary}[corollary]{Corollary}
\aliascntresetthe{corollary}

\newaliascnt{remark}{theorem}

\aliascntresetthe{remark}

\newaliascnt{definition}{theorem}
\newtheorem{definition}[definition]{Definition}
\aliascntresetthe{definition}

\newaliascnt{assumption}{theorem}

\aliascntresetthe{assumption}

\newaliascnt{example}{theorem}
\newtheorem{example}[example]{Example}
\aliascntresetthe{example}

\crefname{theorem}{Theorem}{Theorems}
\Crefname{theorem}{Theorem}{Theorems}
\crefname{lemma}{Lemma}{Lemmas}
\Crefname{lemma}{Lemma}{Lemmas}
\crefname{proposition}{Proposition}{Propositions}
\Crefname{proposition}{Proposition}{Propositions}
\crefname{corollary}{Corollary}{Corollaries}
\Crefname{corollary}{Corollary}{Corollaries}
\crefname{remark}{Remark}{Remarks}
\Crefname{remark}{Remark}{Remarks}
\crefname{definition}{Definition}{Definitions}
\Crefname{definition}{Definition}{Definitions}
\crefname{assumption}{Assumption}{Assumptions}
\Crefname{assumption}{Assumption}{Assumptions}
\crefname{algorithm}{Algorithm}{Algorithms}
\Crefname{algorithm}{Algorithm}{Algorithms}
\crefname{example}{Example}{Examples}
\Crefname{example}{Example}{Examples}

\newcommand{\RR}{\mathbb{R}}
\newcommand{\X}{\mathcal{X}}
\newcommand{\C}{\mathcal{C}}
\newcommand{\dir}{\operatorname{dir}}
\newcommand{\spanop}{\operatorname{span}}
\newcommand{\Wstar}{\dir(\X^\star(\C))}
\newcommand{\dstar}{d^\star}
\newcommand{\argmin}{\mathop{\rm arg\,min}}

\title{Solving Linear and Integer Programs Faster Using Data}
\author{%
  Yuhan Ye \\
  MIT \\
  \texttt{yyh03@mit.edu}
  \And
  Omar Bennouna \\
  MIT \\
  \texttt{omarben@mit.edu}
}
\date{}

\begin{document}
\maketitle

\begin{abstract}

We study how much a linear program (LP) or an integer program (IP) can be compressed when solved repeatedly over a fixed feasible region with varying objective coefficients. We provide a framework that compresses LP and IP formulations with guarantees on the probability of preserving optimality, using either an uncertainty set for the cost vector or historical samples. Our framework exposes a tunable tradeoff between compression and the probability of recovering an optimal solution. We characterize the minimum dimension of affine reformulations preserving the full optimizer set over a given cost region and provide sample-based learning algorithms with fast-rate guarantees, without requiring a nondegeneracy assumption for LPs. With high confidence over the training data, the learned reformulation preserves the optimizer set for an independent test cost with probability at least \(1-\rho-\widetilde{O}(d^\star_\rho/ n)\), where \(\rho\) bounds the excluded probability mass and \(d^\star_\rho\) is the retained region’s decision-relevant dimension. Existing data-driven approaches to LP and IP reduction do not, to our knowledge, learn a single low-dimensional affine reformulation with a prescribed probability of preserving the full optimizer set. Experimental results show substantial improvements over the compared baselines in accuracy, sample efficiency, and computational speed.

\end{abstract}

\section{Introduction}

\begingroup{}
Many decision systems repeatedly solve linear programs (LPs) or integer programs (IPs) over the same feasible
region while the objective changes.  This setting arises in routing, resource
allocation, network operations, and other large-scale planning tasks.  In such
applications, one would like to compress the repeated optimization problem once and
reuse the compressed representation on future instances.  Existing approaches address this goal by learning or sampling low-dimensional projections for LPs \citep{sakaue2024projections,iwata2025heterogeneous,vu2018random}, or by predicting variable fixings and other affine restrictions for IPs \citep{xavier2021learning,nair2020neuraldiving}.
These methods can yield substantial speedups, but may either incur large optimality gaps or require a significant time to learn compressed formulations, and do not guarantee a prescribed probability of recovering an optimal solution.
\endgroup{}

In this paper, we study how to find the lowest-dimensional reformulation of an LP or IP that preserves optimality with a prescribed probability.  Let
\begin{align}
    \min_{x\in\X} c^\top x \label{LP}
\end{align}

be a family of LPs or IPs with a fixed nonempty bounded feasible set $\X$ and a cost vector $c$ known to lie in a prior set $\C\subseteq\RR^d$. For LPs, $\X=\{x\in\RR^d:Ax\le b\}$ is a polytope. For IPs, $\X=\X_I:=\{x\in\mathbb Z^d:Ax\le b\}$ is a bounded polyhedron intersected with \(\mathbb{Z}^d\).
The central question we propose is:

\begin{quote}
When an LP or an IP is solved repeatedly with varying costs, how much can future instances be reduced in dimension while guaranteeing a probability $1-\rho$ of optimality, given prior knowledge about the cost vector?
\end{quote}

\begingroup{}
Here, prior knowledge takes the form of either a set $\mathcal C \subset \mathbb R^d$ containing the cost vector, historical samples of $c$, or both. We first focus on the known-prior-set setting and later extend our results to the case where only historical samples are available. Such problems are common in the optimization literature, with examples including shortest-path routing \citep{delling2017customizable}, min-cost flow \citep{ahuja1995applications,frangioni2006cost}, knapsack \citep{mandi2020smart}, and traveling-salesperson problems \citep{avraham2020datadriven}. More generally, contextual linear and integer optimization models predict objective coefficients from context while keeping the feasible region fixed. In such settings, compression can speed up repeated LP or IP solves while controlling the probability of preserving optimality, and reduce the amount of information that must be learned about $c$.

\endgroup{}

\begingroup{}
Against this background, we show throughout the paper that the quantity governing LP and IP compression is not the ambient dimension $d$, but the dimension of the \emph{decision-relevant subspace}: the subspace spanned by differences between optimal solutions reachable when the cost vector $c$ belongs to a confidence set $\mathcal C$. Crucially, even when $d$ is very large, this subspace can be low-dimensional. Building on
\citet{bennouna2025whatdata} and \citet{bennouna2026datainformativeness}, when $\mathcal{C}$ is
convex and open, this subspace can be written as
\[
   \mathrm{dir}(\mathcal{X}^\star(\mathcal{C}))
    := \mathrm{span}\{x - x' : x, x' \in \mathcal{X}^\star(\mathcal{C})\},
\]
where $\mathcal{X}^\star(\mathcal{C})$ denotes the set of all optimal solutions as $c$ varies
over $\mathcal{C}$. Equivalently, $\Wstar$ is the span of differences between reachable optimal solutions. An important insight from their results is that a set of function evaluations $c^\top q_1, \ldots, c^\top q_r$ contains the necessary information to recover the solution set of~\eqref{LP} if and only if $\Wstar \subseteq \mathrm{span}(q_1, \ldots, q_r)$. Such a set $\{q_1, \ldots, q_r\}$ is called a \emph{sufficient decision dataset} (SDD). The key implication is that the optimal solution depends on $c$ only through its projection onto
$\Wstar$, and $\Wstar$ itself is the smallest subspace with this property.
\endgroup{}

\begingroup{}
Recent work by \citet{ye2026learning} clarifies the computational and statistical landscape of this characterization.  On the negative side, constructing a basis of the decision-relevant space $\Wstar$ is NP-hard and coNP-hard. On the positive side, the same work shows that one can learn decision-relevant directions from i.i.d. cost vector samples. In particular, they provide an algorithm that outputs a set of vectors $\{q_1,\dots,q_r\}$ such that, with high confidence over the training sample, the measurements $c^\top q_1,\dots, c^\top q_r$ have enough information to solve \eqref{LP} for an independent cost drawn from the same distribution with probability at least $1-\widetilde O(\dstar/n)$ where $n$ is the sample size and $d^\star:=\mathrm{dim}\; \Wstar$.
\endgroup{}

\begingroup{}
From the standpoint of LP and IP compression, two challenges remain. First, the existing SDD literature \citep{bennouna2025whatdata,bennouna2026datainformativeness} characterizes which measurements suffice to recover optimal decisions, but it does not provide compressed reformulations of LPs induced by the decision-relevant subspace. Second, the sample-based guarantee for the algorithm in \citet{ye2026learning} assumes a strong nondegeneracy condition, whereas many repeated-LP and IP models of practical interest are degenerate.
\endgroup{}

\subsection{Our contributions}

\begin{enumerate}[leftmargin=1.4em,itemsep=2pt,topsep=2pt]

\item \textbf{Exact reduced LPs and IPs with a fast-rate learning certificate.}
We first study the problem of finding an exact LP or IP reformulation when the cost vector ranges over a known confidence set $\mathcal C$. We turn the decision-relevant subspace characterization of
\citet{bennouna2025whatdata,bennouna2026datainformativeness} into an explicit solver-side
reduction: any subspace containing $\Wstar$ yields an exact reduced LP or IP, and a basis
of $\Wstar$ gives a minimal exact $d^\star$-variable reformulation (\Cref{thm:exact-reduced-lp,prop:ip-affine-exactness}). Unlike
approximate projection methods, which treat the projection as a black-box parameter to be
optimized without identifying what makes it optimal, our framework reveals $\Wstar$
as the unique object that governs exact compression. When the prior set is known,
this reduction requires no historical cost samples, although computing it can be expensive.
With high confidence over the training data, our sample-based algorithms preserve the optimizer set
for an independent test cost drawn from the same distribution with probability at least
$1 - \widetilde{O}(d^\star/n)$. This exactness-failure bound has a $1/n$ dependence, while
approximate methods bound objective-value generalization gaps at rate $\widetilde{O}(1/\sqrt{n})$.
In contrast with the learner of \citet{ye2026learning}, our approach does not require any nondegeneracy assumption (\Cref{thm:adj-cumulative-learning,thm:ip-sample-learning}).

\item \textbf{Sample-based compression at a prescribed probability of optimality.}
We extend our framework to the setting where the uncertainty set $\mathcal C$ is unknown, using historical samples to construct a calibrated confidence set and learn a compressed LP or IP formulation. With high confidence over the training data, the resulting reformulation preserves the optimizer set for an independent test cost with probability at least $1-\rho-\widetilde O(d^\star_\rho/ n)$, where $\rho$ bounds the probability mass outside the confidence set, $d^\star_\rho$ is its decision-relevant dimension (\Cref{thm:calibrated-prior-learning} and its IP extension in \hyperref[item:ip-estimated-prior]{Appendix~\ref*{app:ip-learning}}). This guarantee allows the confidence-set coverage and sample size to be calibrated to a prescribed probability of optimality, exposing a tunable tradeoff between compression and optimality.

\begingroup\item\textbf{Faster learning, fewer training samples, and higher accuracy.}
We compare our method with projection-learning baselines based on
\citet{sakaue2024projections,iwata2025heterogeneous} and a random
variable-selection baseline \citep{akchen2024column} for LPs.
For IPs, we compare with baselines \citep{xavier2021learning,bertsimas2021online,nair2020neuraldiving,li2025fast}.
The experiments cover synthetic and benchmark instances
(Section~\ref{sec:experiments} and Appendix~\ref{app:ip-experiment}).
We show that our exact reduction substantially outperforms the projection learning baselines in terms of data requirement, accuracy, and speed, and recovers the original optimal value more often than the variable-fixing baseline. Our method can also be parallelized when optimizers are unique almost surely, because each training LP or IP can be solved independently. Even without parallelization, our method learns the compression substantially faster than the compared baselines.\par\endgroup \end{enumerate}

\subsection{Related work}
\paragraph{Decision-sufficient representations for linear programs.}
\begingroup{}
\citet{bennouna2025whatdata} and \citet{bennouna2026datainformativeness} characterize what
information about the cost vector is needed to recover optimal decisions. Their key finding is
that optimality is governed by a single geometric object, the subspace $\Wstar$ spanned by differences between reachable optimal
solutions. For open convex priors, a set of linear measurements $\{c^\top q_1, \ldots, c^\top q_r\}$ is sufficient to
solve~\eqref{LP} if and only if the span of the directions $\{q_1, \ldots, q_r\}$ contains $\Wstar$.
\citet{ye2026learning} move to the data-driven setting, where the cost vector is random but
samples are available, and show that directions in $\Wstar$ can be learned
incrementally from cost samples with provable out-of-sample guarantees. Our work builds on this
line of research but shifts the focus from information-theoretic sufficiency to optimization:
we ask how a known $\Wstar$ yields a smaller optimization problem, how to learn this subspace under weaker
geometric assumptions, and how to calibrate a high-coverage working prior when the prior set $\mathcal{C}$ is unknown.
\endgroup{}

\paragraph{Approximate projections for repeated LPs.} Random projections reduce the size of LPs while approximately preserving their objective values \citep{vu2018random,poirion2023randomprojections}. \citet{sakaue2024projections} learn a shared projection for repeated LPs and provide generalization bounds based on pseudo-dimension. \citet{iwata2025heterogeneous} instead learn a projection for each instance, allowing them to handle heterogeneous LP families. Projection methods have also been studied for quadratic programs \citep{dAmbrosio2020,Nguyen2025}. Whereas these LP projection methods target objective-value performance, our learning algorithms are guided by a geometric characterization of the minimum-dimensional affine restriction preserving the full optimizer set over a given cost region. We compare our method with projection baselines in \Cref{sec:experiments}.
\paragraph{Learning reduced IP formulations.} \citet{xavier2021learning} predict which integer variables to fix and which pairs of variables should have the same value, while Neural Diving \citep{nair2020neuraldiving} predicts values for a subset of the integer variables and solves the remaining MIP. Other approaches predict which constraints are tight \citep{li2026constraint} or learn to select a reduced formulation from a finite set of strategies \citep{Bertsimas2021,bertsimas2021online,li2025fast}. \citet[Theorem~3.1]{Bertsimas2021} provide a bound on the probability that a new instance requires an optimal strategy outside this set, but their IP formulations only consider integer assignments seen during training. Our approach instead learns a single affine reformulation that can recover optimal integer solutions not seen during training and provides a guarantee on the probability of preserving the optimizer set for a new cost vector, with a sample requirement that depends on the decision-relevant dimension.

Additional related work is given in \Cref{app:additional-related-work}.

\section{Preliminaries}
We first introduce the notation. Consider the polyhedral LP
\begin{equation}
\label{eq:lp}
    \min_{x\in\X} c^\top x,
    \qquad
    \X:=\{x\in\RR^d:Ax\leq b\},
\end{equation}
\begingroup{}
where the cost vector $c$ lies in a convex prior set $\C\subseteq\RR^d$.
Here, $\X$ is the continuous feasible polytope and $\X^\angle$ denotes its vertices. For IPs, we optimize over $\X_I=\X\cap\mathbb Z^d$, which we assume is nonempty.
For $c\in\RR^d$, set $\X^\star(c):=\argmin_{x\in\X}c^\top x$ and $\X^\star(\C):=\bigcup_{c\in\C}\X^\star(c)$. Following \citet{bennouna2025whatdata,bennouna2026datainformativeness}, $\mathcal D:=\{q_1,\dots q_r\}\subset \mathbb R^d$ is a sufficient decision dataset (SDD) for $(\X,\C)$ if for any $c\in \mathcal C$ the observation of the evaluations $c^\top q_1,\dots,c^\top q_r$ can be mapped to the solution set $\arg\min_{x\in \mathcal X}c^\top x$, i.e. they contain sufficient information to find the solution to \eqref{LP}. \citet{bennouna2025whatdata,bennouna2026datainformativeness} argue that the quantity of interest that determines the information required to solve \eqref{LP} is the decision-relevant space $\Wstar$, which is equal to the span of all directions along which reachable optimal solutions can differ. In particular, they provide the following exact characterization.
\endgroup{}

\begin{theorem}
\label{thm:sdd-subspace}
Assume $\C$ is open and convex. A finite dataset $\mathcal D\subseteq\RR^d$ is an SDD for $(\X,\C)$ if and only if $\Wstar\subseteq\spanop(\mathcal D)$. In particular, the minimum cardinality of an SDD is $\dstar:=\mathrm{dim}\;\Wstar$.
\end{theorem}

\begingroup{}
This characterization has consequences not only from an experimental design perspective, but also for the computational speedup of LPs and IPs.
In particular, we show that a basis for the span of a sufficient dataset, together with a reachable optimizer as anchor, yields a reformulation of \eqref{LP} whose dimension equals the rank of the dataset. The question
we address can be stated rigorously as follows: given an LP or IP of the form
\eqref{LP} with a random cost vector $c$, and given either an uncertainty
set for $c$ or samples from its distribution, how can we construct a
lower-dimensional formulation of the form
\endgroup{}

\begin{equation}
    \min_{y \in \mathcal{X}'} c'^\top y,
\end{equation}
\begingroup{}
where $c'$ and $\mathcal{X}'$ are built from $c$ and $\mathcal{X}$, such
that an optimizer $y$ can be mapped to a solution of \eqref{LP} with high probability over a new cost?
\endgroup{}

\begingroup{}
We show that both settings, samples from the distribution of $c$ and
an uncertainty set for $c$, can be handled within the same framework. The key step in obtaining a lower-dimensional reformulation
of \eqref{LP} is to construct a sufficient dataset for the relevant costs. Two
approaches are available. If an uncertainty set is available and either the LP instance is of moderate size or sufficient computational power is available,
\citet{bennouna2025whatdata} formulate a mixed-integer program (MIP)
that iteratively constructs a sufficient dataset for LPs. When the instance is
too large for an MIP to be tractable or is an IP, and data is available, we instead process cost samples sequentially, learning a new decision-relevant direction only when the current dataset fails to recover the full optimizer set at the incoming sample. The dataset therefore grows only on informative samples and
contains at most $\dstar$ directions.
\endgroup{}

\section{LP and IP Compression via Sufficient-Decision Subspaces}

\subsection{Exact reduction from a sufficient-decision subspace}
\label{subsec:exact-reduction}
The characterization in \cref{thm:sdd-subspace} identifies which cost measurements are sufficient for recovering optimal decisions. \begingroup{}We now use the same geometry on the solver side. Once a subspace contains the decision-relevant directions, the original LP or IP can be solved exactly on a lower-dimensional affine slice anchored at a reachable optimizer.\endgroup{} Proofs and further details of this section are deferred to \Cref{app:proof-exact}.

Fix a prior set $\C$ containing the possible cost vectors. We call an affine subspace $V\subseteq\RR^d$ \emph{exact-sketching} if it preserves the full optimizer set on the prior, namely, for every $c\in\C$,
\[
    \argmin_{x\in\X\cap V} c^\top x
    =
    \argmin_{x\in\X} c^\top x .
\]

\begin{proposition}[Exact LP and IP reduction from a sufficient-decision subspace]
\label{thm:exact-reduced-lp}
An affine subspace $V\subseteq\RR^d$ is exact-sketching on $\C$ if and only if
\[
    \Wstar\subseteq \dir(V)
    \quad\text{and}\quad
    V\cap\X^\star(\C)\ne\emptyset .
\]
\end{proposition}

Equivalently, $V$ is exact-sketching if and only if $\X^\star(\C)\subseteq V$. To construct such a space, choose a reachable optimizer as anchor and include the differences between reachable optimizers.

We next show how an exact-sketching affine subspace yields a potentially lower-dimensional reformulation of \cref{LP}. Fix an anchor $x_0\in\X$ and a matrix $U\in\RR^{d\times r}$ with linearly independent columns. Let $T_{U,x_0}(z):=x_0+Uz$, let $\mathcal Z_{U,x_0}:=\{z\in\RR^r:T_{U,x_0}(z)\in\X\}$ be the compressed feasible region\footnote{In the polyhedral model $\X=\{x\in\RR^d:Ax\le b\}$, $\mathcal Z_{U,x_0}=\{z\in\RR^r:A(x_0+Uz)\le b\}$.}, and let $F_{U,x_0}(c):=\argmin_{z\in\mathcal Z_{U,x_0}}(U^\top c)^\top z$ be the compressed optimizer set. If $x_0\in\X^\star(\C)$ and $\operatorname{range}(U)\supseteq\Wstar$, then \Cref{thm:exact-reduced-lp} applies to the affine space $x_0+\operatorname{range}(U)$, and \begingroup{}the following optimizer-set identity holds for every $c\in\C$:\endgroup{}
\begin{equation}
\label{eq:argmin-reduced-lp}
    \X^\star(c)
    =
    T_{U,x_0}\bigl(F_{U,x_0}(c)\bigr).
    \tag{\(\mathsf{Exact}(U,x_0,c)\)}
\end{equation}
Moreover,
\begin{equation}
\label{eq:exact-reduced-lp}
    \min_{x\in\X} c^\top x
    =
    c^\top x_0+
    \min_{z\in\mathcal Z_{U,x_0}}(U^\top c)^\top z.
\end{equation}
Thus solving the reduced LP in \eqref{eq:exact-reduced-lp} and lifting by $T_{U,x_0}$ recovers an optimal solution of the original LP. For IPs, the same identity holds with $\X_I$ and $\X_I^\star(c)$ in place of $\X$ and $\X^\star(c)$, an anchor $x_0\in\X_I^\star(\C)$, and a subspace containing $\text{dir}(\mathcal{X}_I^\star (\mathcal C))$. The corresponding exactness characterization is given in \Cref{prop:ip-affine-exactness} (Appendix~\ref{app:ip-exact}). The compressed formulation retains the integrality constraint $x_0+Uz\in\mathbb Z^d$.
\begingroup{}
To guarantee that $\mathsf{Exact}(U,x_0,c)$ holds for every $c\in\C$, the minimal choice $\operatorname{range}(U)=\Wstar$ gives an exact $\dstar$-variable affine reformulation. Constructing this minimal subspace, however, can be both NP-hard and coNP-hard \citep{ye2026learning}, even though it can be done in the LP case using the MIP formulation in \citet{bennouna2026datainformativeness}. Our approach aims to find $U$ such that $\mathsf{Exact}(U,x_0,c)$ holds with high probability for an independent cost drawn from the same distribution as the training samples. This allows us to design a tractable algorithm to do so.
\endgroup{}

The next result shows that vertex optimizers are preserved under $\mathsf{Exact}(U,x_0,c)$. This gives integral optimal solutions when the original LP has integral vertices. For general IPs, exactness instead follows from preserving the integer optimizer set while retaining the integrality constraint, as shown in \Cref{prop:ip-affine-exactness} (Appendix~\ref{app:ip-exact}).


\begin{corollary}[Vertex optimizers are preserved]
\label{cor:exact-reduced-vertices}
\label{thm:exact-reduced-vertices}
If $\mathsf{Exact}(U,x_0,c)$ holds, then
\[
    T_{U,x_0}\!\left(
        F_{U,x_0}(c)
        \cap \mathcal Z_{U,x_0}^\angle
    \right)
    =
    \X^\star(c)\cap\X^\angle .
\]
\end{corollary}

\begingroup{}
Thus, under the exactness event $\mathsf{Exact}(U,x_0,c)$, any vertex optimizer of the reduced LP lifts to a vertex optimizer of the original LP. In particular, when the original LP has integral vertices, as in shortest-path formulations with totally unimodular constraints and integral right-hand sides, solving the reduced LP to a vertex optimum recovers an integral optimal solution, and hence IPs that can be relaxed into LPs will not lose their tractability when compressed. For general IPs, the reduced formulation retains the integrality requirement $x_0+Uz\in\mathbb Z^d$, so any reduced optimizer likewise lifts to an integral optimal solution of the original problem.

\endgroup{}

\Cref{ex:affine-vs-linear-projection} illustrates the resulting exact reduction.

\begin{example}[Exact compression onto a reachable optimal edge]
\label{ex:affine-vs-linear-projection}
Let $\X=[-1,1]^2$ and $\C_\rho=(-1-\rho,-1)\times(-1,1)$ for $\rho>0$. Every $c\in\C_\rho$ has $c_1<0$, while costs with $c_2=0$ expose the entire right edge. Thus $\X^\star(\C_\rho)=\{1\}\times[-1,1]$ and $\dir(\X^\star(\C_\rho))=\spanop\{(0,1)\}$. Taking $U=(0,1)^\top$ yields the exact one-variable LP $\min_{z\in[-1,1]}c_2z$, with the original optimal value recovered by adding $c_1$. The map $x=(1,z)$ recovers the full optimizer set, including the entire edge when $c_2=0$. \Cref{fig:affine-vs-linear-projection} illustrates this reduction.
\end{example}

\begin{figure}
    \centering
    \begin{tikzpicture}[scale=1.2,>=stealth, font=\footnotesize]

\begin{scope}
  \draw[->,gray!70] (-1.8,0) -- (1.8,0) node[right] {$c_1$};
  \draw[->,gray!70] (0,-1.8) -- (0,1.8) node[above] {$c_2$};

  \fill[blue!12] (-1.7,-1) rectangle (-1,1);
  \draw[blue!60, thick, dashed] (-1.7,-1) rectangle (-1,1);

  \draw[blue!60, thick, dashed] (-1,-1) -- (-1,1);

  \draw[<->, blue!60!black, thin] (-1.7,-1.55) -- (-1,-1.55);
  \node[blue!60!black, below] at (-1.35,-1.6) {$\rho$};

  \draw[gray!70] (-1,0.05) -- (-1,-0.05);
  \draw[gray!70] (-1.7,0.05) -- (-1.7,-0.05);
  \node[gray!50!black, anchor=north] at (-1,-1.05) {$-1$};
  \node[gray!50!black, anchor=north] at (-1.7,-1.05) {$-1\!-\!\rho$};
  \draw[gray!70] (0.05,1) -- (-0.05,1);
  \node[gray!50!black, anchor=west] at (0.07,1) {$1$};
  \draw[gray!70] (0.05,-1) -- (-0.05,-1);
  \node[gray!50!black, anchor=west] at (0.07,-1) {$-1$};

  \node[blue!70!black] at (-0.65,0.3) {$\mathcal{C}_\rho$};
  \node at (0,-2.0) {\textbf{Cost space}};
\end{scope}

\begin{scope}[xshift=6.2cm]
  \draw[->,gray!70] (-1.8,0) -- (2.4,0) node[right] {$x_1$};
  \draw[->,gray!70] (0,-1.8) -- (0,2.2) node[above] {$x_2$};
  \fill[gray!10] (-1,-1) rectangle (1,1);
  \draw[thick] (-1,-1) rectangle (1,1);
  \node at (-0.65,0.65) {$\X$};
  \draw[red, line width=2pt] (1,-1) -- (1,1);
  \draw[->,red,thick] (0.55,-0.35) -- (0.55,0.4) node[midway,left] {$U$};
  \node[red,anchor=west] at (1.08,0.5) {$x=(1,z)$};
  \node[red,anchor=west] at (1.08,-0.5) {$\X^\star(\C_\rho)$};
  \draw[gray!70] (1,0.05) -- (1,-0.05);
  \draw[gray!70] (-1,0.05) -- (-1,-0.05);
  \node[gray!50!black,anchor=north] at (1,-1.05) {$1$};
  \node[gray!50!black,anchor=north] at (-1,-1.05) {$-1$};
  \draw[gray!70] (0.05,1) -- (-0.05,1);
  \draw[gray!70] (0.05,-1) -- (-0.05,-1);
  \node[gray!50!black,anchor=east] at (-1.05,1) {$1$};
  \node[gray!50!black,anchor=east] at (-1.05,-1) {$-1$};
  \filldraw (0,0) circle (0.8pt);
  \node[below left=-1pt] at (0,0) {$0$};
  \node at (0.3,-2.2) {\textbf{Decision space}};
\end{scope}

\end{tikzpicture}
    \caption{Exact compression through the decision-relevant subspace. Costs in $\C_\rho$ have $c_1<0$, so all reachable optimizers lie on the highlighted right edge. The direction $U=(0,1)^\top$ parameterizes this edge as $x=(1,z)$, reducing the LP to one variable while preserving its full optimizer set.}
    \label{fig:affine-vs-linear-projection}
\end{figure}

\begingroup{}
\subsection{Learning exact compression from samples}
\label{subsec:adj-learning} 
We now describe the sample-based learning algorithm (\cref{alg:adj-cumulative}).\footnote{The algorithm and tractability statement below are for LPs. The IP counterpart is given in \Cref{alg:ip-cumulative,thm:ip-sample-learning} (Appendix~\ref{app:ip-learning}).}
It maintains a matrix $U$ of decision-relevant directions discovered so far. Together with a fixed anchor $x_0\in\X^\star(\C)$, obtained by picking any $c_0\in\C$ and taking $x_0\in\X^\star(c_0)$, it defines the candidate affine slice $x_0+\operatorname{range}(U)$. Training costs $c_1,\dots,c_n$ are processed one at a time, and $U$ grows whenever a new sample reveals a direction the current slice misses. For each $c_i$, the algorithm checks whether the optimal face $\X^\star(c_i)$ is contained in the current slice $x_0+\operatorname{range}(U)$. If yes, it proceeds to $c_{i+1}$. If no, then since $\X^\star(c_i)$ is the convex hull of its vertices, some vertex $x$ lies outside the slice. The algorithm appends $x-x_0$ as a new column of $U$ and rechecks. This test-and-append loop repeats until the full optimal face is contained. Samples that trigger at least one append are called \emph{hard instances}.
\endgroup{}


\begingroup{}
\begin{algorithm}[t]
\caption{Learning exact LP compression}
\label{alg:adj-cumulative}
\begin{algorithmic}[1]
\Require Feasible polytope $\X$, fixed reachable anchor $x_0\in\X^\star(\C)$, training samples $c_1,\ldots,c_n$, and deterministic tie-breaking for all choices.
\Ensure A compression matrix $U_n$ and a hard-sample index set $T_n$.
\State Initialize $U_0\gets [\,]\in\RR^{d\times 0}$ and $T_n\gets\emptyset$.
\For{$i=1,\ldots,n$}
    \State Set $U_i\gets U_{i-1}$.
    \While{$\X^\star(c_i)\nsubseteq x_0+\operatorname{range}(U_i)$}
        \State Choose $x\in\X^\star(c_i)\cap\X^\angle$ with $x-x_0\notin\operatorname{range}(U_i)$.
        \State $U_i\gets [\,U_i\;\;x-x_0\,]$ and $T_n\gets T_n\cup\{i\}$.
    \EndWhile
\EndFor
\State \Return $U_n$ and $T_n$.
\end{algorithmic}
\end{algorithm}
\endgroup{}

Let $P_c$ denote the cost distribution supported on $\C$. In the following guarantee, the confidence $1-\delta$ is over the i.i.d. training sample. Conditional on the learned $U_n$, the probability in \eqref{eq:fast-rate} is over an independent test cost $c\sim P_c$.

\begin{theorem}[Exact-compression certificate and tractability]
\label{thm:adj-cumulative-learning}
Fix an anchor $x_0\in\X^\star(\C)$ independently of the $n$ training costs.  Run \Cref{alg:adj-cumulative} on $n$ i.i.d. costs drawn from a distribution supported on $\C$.  Then
\[
    \operatorname{range}(U_n)\subseteq\Wstar,
    \qquad
    |T_n|\le \operatorname{rank}(U_n)\le\dstar,
\]
and $\mathsf{Exact}(U_n,x_0,c_i)$ holds for every training sample $i=1,\ldots,n$.  Moreover, for any $\delta\in(0,1)$, with probability at least $1-\delta$,
\begin{equation}
\label{eq:fast-rate}
    \mathbb P_{c\sim P_c}\!\left[\mathsf{Exact}(U_n,x_0,c)\right]
    \ge
    1-\frac{4}{n}\bigl(6|T_n|+\log(e/\delta)\bigr)
    \ge
    1-\frac{4}{n}\bigl(6\dstar+\log(e/\delta)\bigr).
\end{equation}

\end{theorem}

\paragraph{Implementation and parallelism.}
When $U$ has rank $r$, each containment test uses one LP to find the optimal value and at most $2(d-r)$ auxiliary LPs over the optimal face. A failed test supplies a new extreme optimizer (\Cref{lem:adj-lp-containment}). With at most $n+\dstar$ tests, the algorithm runs in polynomial time in the LP-oracle model. If optimizers are unique almost surely, as for an absolutely continuous cost distribution, we can instead solve each training LP once and compute a basis of $\spanop\{x_i-x_0:i=1,\ldots,n\}$ using the same independent anchor $x_0$. Beyond the anchor, this requires only $n$ independent LP solves, all parallelizable. By comparison, the gradient-based methods of \citet{sakaue2024projections,iwata2025heterogeneous} require $En$ projected-LP solves over $E$ training passes, in addition to any initialization or normalization solves. Although computations within a batch can be parallelized, successive parameter updates create dependencies between batches. Our construction avoids this iterative training, requiring only the independent solves and a basis computation.

\begingroup{}
\paragraph{Sample-free construction and the computational tradeoff.} The sample-based learner is complementary to the sample-free alternative. When $\C$ is known, \citet{bennouna2025whatdata} construct a global SDD directly from the pair $(\X,\C)$ using an iterative construction with a mixed-integer implementation, requiring no historical cost samples. This construction can be computationally expensive. \citet{ye2026learning} show that, for open convex $\C$, constructing a minimum-size global SDD and even computing its minimum size are NP-hard and coNP-hard. Our sample-based learner trades this combinatorial hardness for a statistical cost. Each update requires only LP or IP solves (depending on whether the problem at hand is an LP or IP), optimal-face containment tests, and rank tests, but the algorithm needs representative cost samples to discover the relevant directions. Once the sampled directions span $\Wstar$, $\mathsf{Exact}(U_n,x_0,c)$ holds for every $c\in\C$. This follows from \Cref{thm:exact-reduced-lp} for LPs and \Cref{prop:ip-affine-exactness} (Appendix~\ref{app:ip-exact}) for IPs. Before this occurs, \Cref{thm:adj-cumulative-learning} gives the LP failure bound, and the IP counterpart is \Cref{thm:ip-sample-learning} (Appendix~\ref{app:ip-learning}).
\endgroup{}


\begingroup{}
\subsection{Ruling out instances that are too hard}
\label{subsec:prior-estimation}
\endgroup{}
\begingroup{}
\cref{alg:adj-cumulative} does not require knowledge of $\mathcal{C}$. Seeking an exact reformulation of (\ref{LP}) for $c$ ranging over the \emph{entire} support of its distribution can be too conservative. A natural way to address this is to rule out training instances that are \emph{too hard} to accommodate in a small reformulation, because the extra directions they require are needed only on a low-probability part of the cost distribution. This also allows us to potentially trade optimality for LP or IP compression by only learning the most important decision-relevant directions. To do so, we construct a convex estimated prior $\widehat\C$ from the data, and only learn the relevant decision directions for the cost vector samples that are inside this set, which we define in the following.
\endgroup{}

\begin{definition}[$(\rho,\delta_0)$-estimated prior]
\label{def:estimated-prior}
Let $P_c$ be the cost distribution.  A pilot-measurable random convex set $\widehat\C\subseteq\RR^d$ is a $(\rho,\delta_0)$-estimated prior if
\[
    \mathbb P_{\rm pilot}\left[P_c(c\notin\widehat\C)\le \rho\right]
    \ge 1-\delta_0 .
\]
\end{definition}

We next state the learning guarantee for this new framework.  Let $\widehat\C$ be a convex $(\rho,\delta_0)$-estimated prior\footnote{The distribution-free construction for such a prior and its proof are deferred to \Cref{app:estimated-prior-construction}.}, and define $\mathcal E_0:=\{P_c(c\notin\widehat\C)\le\rho\}$.  On $\mathcal E_0$, set $P_{\widehat\C}:=P_c(\cdot\mid c\in\widehat\C)$ and $\widehat d_\star:=\dim\dir(\X^\star(\widehat\C))$.  Choose an anchor $\widehat x_0\in\X^\star(\widehat\C)$ by selecting any $\widehat c_0\in\widehat\C$ and solving the original LP.  Draw $n\ge1$ i.i.d. training costs independently of the pilot and retain the $\hat n$ that fall inside $\widehat\C$.\footnote{When the prior's coverage is close to $1-\rho$, $\hat n\approx(1-\rho)n$. Appendix~\ref{app:retained-sample-size} gives the precise relation and a concentration bound.} Run \Cref{alg:adj-cumulative} with anchor $\widehat x_0$ on these retained costs in their original order, and denote its output by $U_{\hat n}$.


\begin{theorem}[Learning with a calibrated estimated prior]
\label{thm:calibrated-prior-learning}
Under the preceding setup, for any $\delta_1\in(0,1)$, with probability at least $1-\delta_0-\delta_1$ over the pilot and the $n$ training costs, the following bound holds whenever $\hat n\ge1$:
\begin{equation}
\label{eq:calibrated-prior-risk}
    \mathbb P_{c\sim P_c}\!\left[
        \mathsf{Exact}(U_{\hat n},\widehat x_0,c)
    \right]
    \ge
    1-\rho-
    \frac{4}{\hat n}\bigl(6\widehat d_\star+\log(e/\delta_1)\bigr).
\end{equation}
Consequently, taking $\widehat\C=\widehat\C_{\rho,\delta_0}$ from \Cref{prop:tolerance-estimated-prior} gives the same distribution-free unknown-prior certificate.
\end{theorem}

\begingroup{}
The first term, $\rho$, bounds the probability mass left outside the calibrated estimated prior, while the second term bounds the stable-compression risk for costs drawn from the conditional distribution inside the realized prior. The parameter $\rho$ controls the degree to which low-probability instances may be discarded. A larger $\rho$ permits a smaller prior and can reduce its decision-relevant dimension $\widehat{d}_\star$, which would result in a more compressed LP reformulation. This lets us calibrate $\rho$ to prioritize optimality (small $\rho$) or compression (large $\rho$). In some cases, a small increase in $\rho$ can substantially reduce the number of variables in the reduced LP. We examine this tradeoff in \Cref{sec:experiments}.
\endgroup{}

\section{Numerical Experiments}\label{sec:experiments}
We evaluate the sample-based exact-compression framework in the unknown-prior setting of \Cref{subsec:prior-estimation}.\footnote{\raggedright{Code and results for the LP and IP experiments are available at \url{https://github.com/yeyuhanyyh/Solving-Linear-and-Integer-Programs-Faster-Using-Data}.}}  The learner first builds a convex working prior from historical cost samples and then learns the decision-relevant directions inside that set. The main text reports the two diagnostics that directly test this mechanism: how performance changes as the target outside-mass level $\rho$ is varied, and how our method compares with a neural projection baseline as the number of training costs changes.  We also compare the time needed by our method and the baselines to learn a compression and to solve LPs using the learned compression. Additional experiments and details are reported in \Cref{app:beyond-prior-experiment}.

\paragraph{Experimental protocol.}
We compare our exact-compression method, denoted \textsc{Ours}, with two projection-learning baselines.  The first is \textsc{DataDrivenProj}, a data-driven projection method motivated by \citet{sakaue2024projections}: it learns a shared $K$-dimensional projection from training LP instances and then solves the projected LP at test time.  The second is \textsc{FCNN}, a cost-only fully connected neural network baseline following the neural projection paradigm of \citet{iwata2025heterogeneous}: it maps the observed cost vector to an instance-specific projection and solves the resulting reduced LP.
We also include \textsc{Rand}, a random projection baseline that selects $K$ variables uniformly at random in the preprocessed LP and fixes the others to zero \citep{akchen2024column}.

In the $\rho$-sweep, we construct an estimated prior of the form $\widehat\C=\{c:s_{\widehat\theta}(c)\le S_{(k)}\}$, where $s_{\widehat\theta}$ is the ridge Mahalanobis score from \Cref{app:estimated-prior-construction} and $S_{(k)}$ is the calibration order statistic associated with the target outside-mass level $\rho$.  The radius is therefore calibrated from empirical scores.  After forming $\widehat\C$, we run \Cref{alg:adj-cumulative} on the retained training costs to obtain an affine exact-compression matrix of learned dimension $d(\rho)$. We use the synthetic repeated-LP families and selected Netlib instances used throughout the paper\footnote{Details of the LP instances and cost distributions are deferred to Appendix~\ref{app:beyond-prior-experiment}.} as a benchmark.  Accuracy is measured by the average test objective ratio, i.e., the objective value returned by a method divided by the full-LP objective value on the same test instance. Each plotted setting uses twelve held-out test LPs. Shaded regions indicate uncertainty around the plotted means.

\paragraph{Compression--coverage tradeoff through $\rho$.}
\begin{figure*}[t]
\centering
\includegraphics[width=0.97\textwidth]{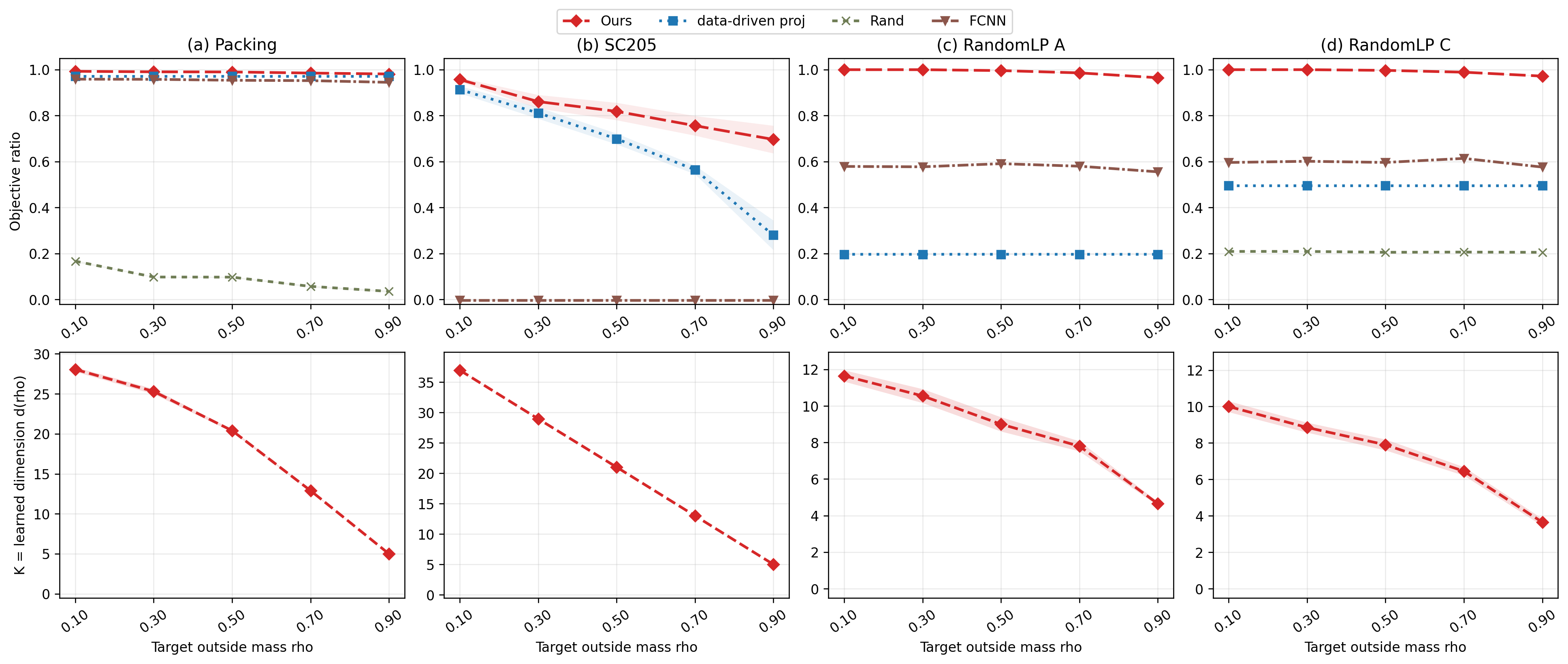}
\caption{Comparison across target outside-mass levels. Rand corresponds to random projections.}
\label{fig:beyond-prior-dynamic-k-rho}
\end{figure*}

\Cref{fig:beyond-prior-dynamic-k-rho} varies the target outside-mass level $\rho$.  For each $\rho$, our approach constructs a different calibrated set $\widehat\C_\rho$, learns a compression of dimension $d(\rho)$, and solves the corresponding affine reduced LP.  The projection baselines are then given the same dimensional budget $K=d(\rho)$, so all methods use the same number of reduced LP variables.

This experiment illustrates the main message of \Cref{subsec:prior-estimation}.  Increasing $\rho$ intentionally discards more low-probability cost mass, which can lower $d(\rho)$ and hence produce a smaller reduced LP.  The empirical curves show that this compression can be obtained with little loss on many test instances, while the projection baselines often fail to match the exact method even at the same dimension.  In other words, the estimated prior does not merely choose a smaller $K$. It focuses learning on optimizer differences within the estimated prior.

\paragraph{Comparison against neural projections.}
\begin{figure*}[t]
\centering
\includegraphics[width=0.97\textwidth]{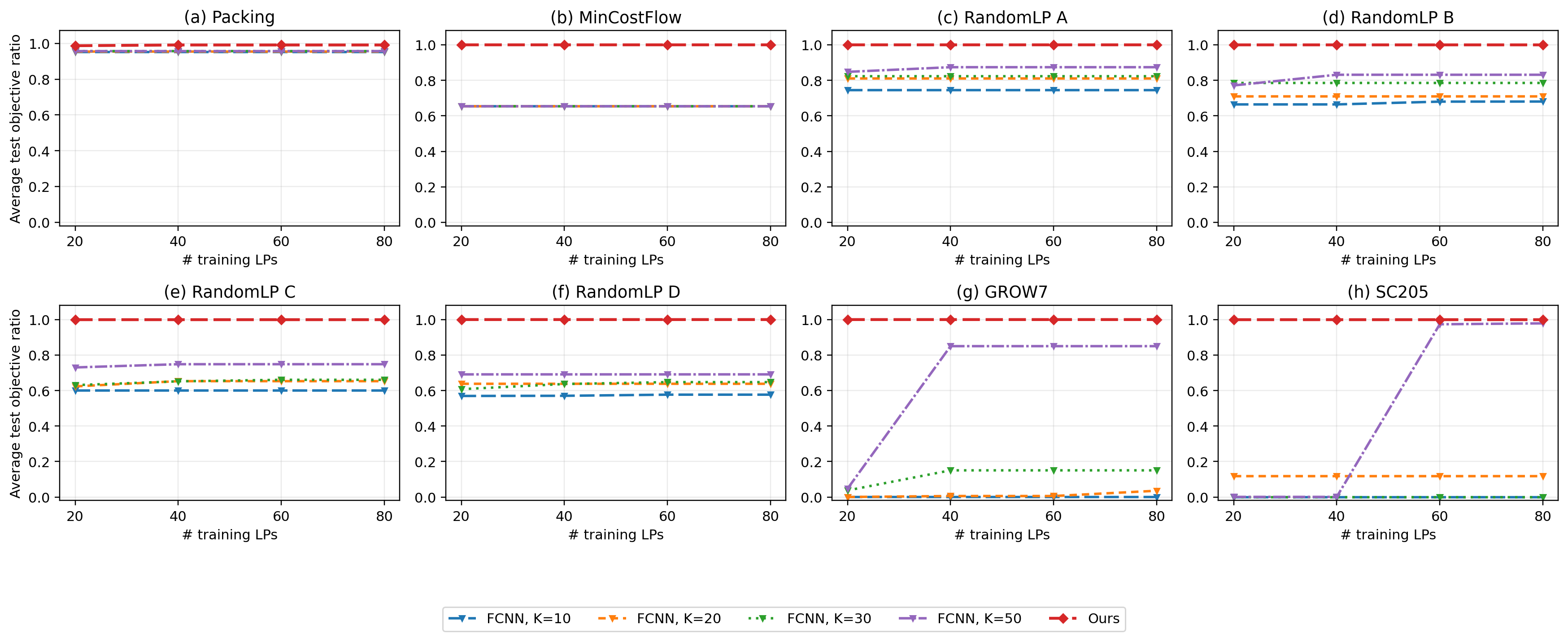}
\caption{Comparison against the \textsc{FCNN} baseline.}
\label{fig:beyond-prior-costonly-samples}
\end{figure*}

\Cref{fig:beyond-prior-costonly-samples} compares our method with the \textsc{FCNN} baseline as the number of training costs varies.  For \textsc{FCNN}, each dashed curve reports the objective ratio using the displayed number of training samples and one of the reduced dimensions $K\in\{10,20,30,50\}$.  For this comparison, \textsc{Ours} uses the same training costs to fit the prior and learn the compression, with the largest training score as the threshold (Appendix~\ref{app:additional-beyond-prior-diagnostics}). These results indicate that our algorithm works substantially better than the compared baseline.

\begin{table}[!htbp]

\centering
\fontsize{8}{9.6}\selectfont
\renewcommand{\arraystretch}{1.15}
\medmuskip=1mu
\setlength{\tabcolsep}{1pt}
\scalebox{0.94}{%
\begin{tabular*}{1.06383\linewidth}{@{\extracolsep{\fill}}lrr*{7}{r}@{}}
\toprule
& & & \multicolumn{3}{c}{Learning (s)} & \multicolumn{4}{c}{LP solve (ms)}\\
\cmidrule(lr){4-6}\cmidrule(lr){7-10}
Instance & $d$ & $K$ & \textbf{\textsc{Ours}} & \shortstack{\textsc{DataDriven}\\\textsc{Proj}} & FCNN & Full & \textbf{\textsc{Ours}} & \shortstack{\textsc{DataDriven}\\\textsc{Proj}} & FCNN\\
\midrule
Packing & 360 & 21 & \scalebox{0.94}{$\bm{0.104\pm 0.006}$} & \scalebox{0.94}{$3.314\pm 0.106$} & \scalebox{0.94}{$37.998\pm 0.247$} & \scalebox{0.94}{$3.381\pm 0.112$} & \scalebox{0.94}{$\bm{2.720\pm 0.067}$} & \scalebox{0.94}{$9.192\pm 1.278$} & \scalebox{0.94}{$5.755\pm 0.347$}\\
RandomLP A & 140 & 9 & \scalebox{0.94}{$\bm{0.246\pm 0.001}$} & \scalebox{0.94}{$0.913\pm 0.039$} & \scalebox{0.94}{$2.864\pm 0.010$} & \scalebox{0.94}{$10.790\pm 0.685$} & \scalebox{0.94}{$\bm{1.944\pm 0.031}$} & \scalebox{0.94}{$5.227\pm 0.303$} & \scalebox{0.94}{$5.655\pm 1.693$}\\
RandomLP C & 220 & 10 & \scalebox{0.94}{$\bm{0.690\pm 0.001}$} & \scalebox{0.94}{$1.955\pm 0.250$} & \scalebox{0.94}{$4.929\pm 0.077$} & \scalebox{0.94}{$30.216\pm 0.082$} & \scalebox{0.94}{$\bm{2.893\pm 0.062}$} & \scalebox{0.94}{$11.809\pm 1.428$} & \scalebox{0.94}{$9.286\pm 2.038$}\\
SC205 & 203 & 21 & \scalebox{0.94}{$\bm{0.110\pm 0.004}$} & \scalebox{0.94}{$6.632\pm 0.095$} & \scalebox{0.94}{$18.635\pm 0.243$} & \scalebox{0.94}{$4.273\pm 0.527$} & \scalebox{0.94}{$\bm{4.251\pm 0.219}$} & \scalebox{0.94}{$4.740\pm 0.116$} & \scalebox{0.94}{$2.336\pm 0.237$}\\
\bottomrule
\end{tabular*}}
\setlength{\abovecaptionskip}{5pt}
\setlength{\belowcaptionskip}{0pt}
\caption{Timing on the \Cref{fig:beyond-prior-dynamic-k-rho} instances at $\rho=0.5$, using $n=40$ training costs before prior filtering. Here $d$ and $K$ count original and reduced LP variables, respectively. All compressed methods use the same $K$. Entries are means $\pm$ sample standard deviations over three timing repeats.}
\label{tab:runtime-rho}

\end{table}

\begin{table}[!htbp]

\centering
\fontsize{8}{9.6}\selectfont
\renewcommand{\arraystretch}{1.15}
\medmuskip=1mu
\setlength{\tabcolsep}{1pt}
\scalebox{0.94}{%
\begin{tabular*}{1.06383\linewidth}{@{\extracolsep{\fill}}lrcrrrrr@{}}
\toprule
& & & \multicolumn{2}{c}{Learning (s)} & \multicolumn{3}{c}{LP solve (ms)}\\
\cmidrule(lr){4-5}\cmidrule(lr){6-8}
Instance & $d$ & $K$& \textbf{\textsc{Ours}} & FCNN & Full & \textbf{\textsc{Ours}} & FCNN\\
\midrule
Packing & 360 & $(19,19,19)$ & $\bm{0.084\pm 0.005}$ & $40.066\pm 1.766$ & $7.208\pm 0.387$ & $\bm{3.204\pm 0.136}$ & $11.249\pm 0.931$\\
MinCostFlow & 360 & $(3,1,1)$ & $\bm{0.189\pm 0.017}$ & $25.973\pm 5.973$ & $12.616\pm 0.497$ & $\bm{1.890\pm 0.351}$ & $3.927\pm 1.025$\\
RandomLP A & 140 & $(5,11,14)$ & $\bm{0.276\pm 0.016}$ & $25.985\pm 5.165$ & $18.761\pm 2.068$ & $\bm{2.789\pm 0.478}$ & $6.933\pm 0.997$\\
RandomLP B & 180 & $(10,7,8)$ & $\bm{0.459\pm 0.021}$ & $29.022\pm 3.529$ & $31.295\pm 2.158$ & $\bm{2.774\pm 0.465}$ & $6.703\pm 1.666$\\
RandomLP C & 220 & $(5,7,3)$ & $\bm{0.810\pm 0.050}$ & $28.304\pm 3.920$ & $48.423\pm 4.036$ & $\bm{2.765\pm 0.297}$ & $5.462\pm 0.636$\\
RandomLP D & 260 & $(7,10,5)$ & $\bm{1.197\pm 0.056}$ & $36.051\pm 5.238$ & $67.593\pm 1.726$ & $\bm{3.433\pm 0.605}$ & $7.719\pm 2.428$\\
GROW7 & 301 & $(19,19,19)$ & $\bm{0.378\pm 0.020}$ & $60.714\pm 0.634$ & $25.622\pm 1.438$ & $\bm{6.870\pm 0.431}$ & $16.727\pm 3.412$\\
SC205 & 203 & $(1,1,1)$ & $\bm{0.114\pm 0.007}$ & $14.193\pm 0.578$ & $7.243\pm 0.369$ & $\bm{1.638\pm 0.094}$ & $2.763\pm 0.297$\\
\bottomrule
\end{tabular*}}
\setlength{\abovecaptionskip}{5pt}
\setlength{\belowcaptionskip}{0pt}
\caption{Timing on the \Cref{fig:beyond-prior-costonly-samples} instances with $n=20$ training costs and 20 test LPs per seed. Entries are means $\pm$ sample standard deviations across three independent seeds. Each $K$ triple gives the number of  LP variables for these seeds. \textsc{Ours} and FCNN use the same $K$.}
\label{tab:runtime-matched-k}

\end{table}

\paragraph{Learning and LP solving time.}
Tables~\ref{tab:runtime-rho} and~\ref{tab:runtime-matched-k} report learning and LP solving times for the instances in \Cref{fig:beyond-prior-dynamic-k-rho,fig:beyond-prior-costonly-samples}, respectively. We report learning time separately from the solve time per test LP.
As discussed in \Cref{subsec:adj-learning}, a major advantage of our method is the low computational cost of learning the compression. The stochastic gradient implementation of \textsc{DataDrivenProj} makes one pass through the training set. Each update solves one projected LP and $K$ quadratic programs (QPs) for the projection step. \textsc{FCNN} makes multiple passes through the training set, solving one projected LP for each sample in each pass. Our method avoids these QPs and iterative neural network training.

Specifically, for each problem and seed in Table~\ref{tab:runtime-matched-k}, our method and \textsc{FCNN} use the same 20 training LPs. Our method solves each training LP once and computes a basis of the optimizer differences. In these runs, \textsc{FCNN} makes 120 passes through the training set, requiring $20 \times 120 = 2{,}400$ projected-LP solves. It also evaluates six separate validation LPs before training and after each pass, requiring $6 \times 121 = 726$ projected-LP solves. In these experiments, our approach is $30.1\times$--$477.0\times$ faster to learn than \textsc{FCNN}. Its reduced LP is faster than both the full LP and \textsc{FCNN} on all eight matched-$K$ instances.

For the highly sparse Packing instance in Table~\ref{tab:runtime-matched-k}, matching the reduced dimension $K$ across methods does not imply equally sparse reduced constraint matrices. The original constraint matrix has 720 nonzeros (0.48\% density), our compressed constraint matrix has 2,090 nonzeros (26.19\% density), and the \textsc{FCNN} compressed constraint matrix is essentially dense, with approximately 7,980 nonzeros (approximately 100\% density). Thus, at the same reduced dimension, our representation retains substantially more sparsity than \textsc{FCNN}. Together with dimension reduction, this sparsity advantage helps explain our method's shorter solve time. A smaller reduced dimension alone need not yield a faster LP solve.

For IPs, we also evaluate our method on unit commitment and knapsack problems. Our method achieves higher recovery rates and objective ratios than the other compression methods. It also learns much faster than the baselines \citep{xavier2021learning,bertsimas2021online,nair2020neuraldiving,li2025fast} (Appendix~\ref{app:ip-experiment}, Tables~\ref{tab:ip-matched-uc}--\ref{tab:ip-random-knapsack}).

\section{Conclusion}
\label{sec:conclusion} 
In this paper, we showed how to compress LPs and IPs that are solved repeatedly with varying costs. We characterized the smallest affine reformulation that preserves the full optimizer set over a given cost region and provided sample-based algorithms with fast-rate guarantees, without requiring nondegeneracy. When the cost region is unknown, we estimate it from historical samples, allowing a tradeoff between compression and the probability of optimality. When optimizers are unique almost surely, each training problem only needs to be solved once, and these solves can run in parallel. Our experiments show substantial improvements over the compared baselines in accuracy, sample efficiency, and computational speed. Future work includes improving the sparsity of the reduced formulations and extending the approach to problems whose feasible regions also change, as well as more general optimization problems.

\clearpage
\begin{ack}
Yuhan Ye is supported by the MIT Schwarzman College of Computing PhD Fellowship, and Omar Bennouna is supported by the MathWorks Fellowship.
\end{ack}

\bibliographystyle{plainnat}
\begingroup{}
\bibliography{reference}
\endgroup{}

\appendix

\begingroup{}
\section{Additional Related Work}
\label{app:additional-related-work}
\paragraph{Parametric and multiparametric optimization.}
Repeated optimization with varying parameters has long been studied through parametric programming and sensitivity analysis \citep{Fiacco1983}. In control and explicit model predictive control, a classical strategy is to precompute the full piecewise-affine solution map offline through multiparametric programming \citep{Bemporad2002,Alessio2009,Borrelli2017}. These methods target an explicit representation of the optimizer over a parameter region. Our goal is different. We do not attempt to compute the entire solution map, but rather to identify the decision-relevant subspace spanned by differences between reachable optimizers and to use it for exact low-dimensional reformulation.

\paragraph{Structural LP reductions.}
\citet{grohe2014dimension} use colour refinement to reduce the numbers of variables and constraints while preserving the optimal value. We study a common affine restriction that preserves the full optimizer set over a region of cost vectors.

\paragraph{Learning to accelerate optimization.}
A broader recent literature uses data to accelerate future solves without explicitly learning a decision-sufficient subspace. Representative directions include learning warm starts for iterative solvers \citep{Sambharya2024}, predicting active sets or optimization strategies \citep{Bertsimas2021,Misra2022}, and amortized optimization methods that map instances directly to approximate solutions \citep{Amos2022,Chen2022}. These approaches can substantially reduce computation, but they typically do not identify a task-dependent exact low-dimensional LP together with a sufficiency guarantee.

\paragraph{Data-driven algorithm design and generalization guarantees.}
\citet{gupta2017pac}, \citet{gupta2020datadrivenalgdesign}, \citet{balcan2017learningtheoreticfoundations}, \citet{balcan2020datadrivenalgdesignchapter}, and \citet{balcan2024howmuchdata} formalize the broader viewpoint that one can learn algorithmic components from representative instances and still obtain out-of-sample guarantees. The LP-projection works above fit this template, with sample complexity bounds obtained via uniform-convergence arguments based on pseudo-dimension. Our framework departs from this approach in two ways. First, the object being learned is different. Instead of a generic projection matrix, we learn a specific geometric object, the decision-relevant subspace, and our algorithm grows this subspace one direction at a time, only when a training sample exposes a direction the current subspace misses. Second, the proof technique is different. This incremental structure makes our learner a \emph{stable compression scheme} in the sense of \citet{hanneke2021stablecompression}, which yields a fast $\widetilde O(\dstar/n)$ bound on exactness failure. The uniform-convergence bounds for projection methods instead concern objective-value generalization gaps and scale as $\widetilde O(1/\sqrt n)$ \citep{sakaue2024projections,iwata2025heterogeneous}. In our unknown-prior guarantee, we separate the samples spent on calibrating the working prior from those spent on discovering decision-relevant directions.

\endgroup{}

\section{Deferred proofs}
\label{app:deferred-proofs}

\subsection{Proofs for Subsection~\ref{subsec:exact-reduction}}
\label{app:proof-exact}

\begin{proof}[Proof of \Cref{thm:exact-reduced-lp}]
Suppose first that $V$ is exact-sketching on $\C$.  Then every reachable optimizer belongs to $V$: if $v\in\X^\star(\C)$, then for some $c\in\C$ we have $v\in\argmin_{x\in\X}c^\top x$, and exactness gives $v\in\argmin_{x\in\X\cap V}c^\top x\subseteq V$.  Hence $V\cap\X^\star(\C)\ne\emptyset$.  Moreover, for any $v,v'\in\X^\star(\C)$, both points lie in $V$, so $v-v'\in\dir(V)$.  Taking spans over all such pairs yields $\Wstar\subseteq\dir(V)$.

Conversely, assume $\Wstar\subseteq\dir(V)$ and choose $x_0\in V\cap\X^\star(\C)$.  For every $v\in\X^\star(\C)$, we have $v-x_0\in\Wstar\subseteq\dir(V)$, so $v\in V$.  Thus $V$ contains all reachable optimizers.  Fix $c\in\C$.  The optimal face $\argmin_{x\in\X}c^\top x$ is the convex hull of its extreme points, and these extreme points belong to $\X^\star(\C)$.  Therefore the whole optimal face is contained in $V$. Hence, the restricted and original LPs have the same optimizer set, and $V$ is exact-sketching on $\C$.
\end{proof}

\begin{proof}[Justification of \eqref{eq:exact-reduced-lp} and \eqref{eq:argmin-reduced-lp}]
Let $V:=x_0+\operatorname{range}(U)$.  Under the assumptions stated before \eqref{eq:argmin-reduced-lp}, \Cref{thm:exact-reduced-lp} gives
\[
    \argmin_{x\in\X\cap V}c^\top x=\X^\star(c),\qquad c\in\C.
\]
The map $T_{U,x_0}$ is an affine bijection from $\mathcal Z_{U,x_0}$ onto $\X\cap V$, and
$c^\top T_{U,x_0}(z)=c^\top x_0+(U^\top c)^\top z$.  Therefore
\[
    T_{U,x_0}\bigl(F_{U,x_0}(c)\bigr)
    =
    \argmin_{x\in\X\cap V}c^\top x
    =
    \X^\star(c),
\]
which is \eqref{eq:argmin-reduced-lp}.  The same parametrization gives
\[
    \min_{x\in\X}c^\top x
    =\min_{x\in\X\cap V}c^\top x
    =c^\top x_0+
      \min_{z\in\mathcal Z_{U,x_0}}(U^\top c)^\top z,
\]
which is \eqref{eq:exact-reduced-lp}.
\end{proof}

Write $V_{U,x_0}=x_0+\operatorname{range}(U)$ for the affine slice.

\begin{proof}[Proof of \Cref{cor:exact-reduced-vertices}]
Let $F_Z:=F_{U,x_0}(c)$ and $F_X:=\X^\star(c)$.  Exactness gives $T_{U,x_0}(F_Z)=F_X$.  Since $T_{U,x_0}$ is an affine bijection from $\mathcal Z_{U,x_0}$ onto $\X\cap V_{U,x_0}$, it is also an affine bijection from the face $F_Z$ onto the face $F_X$.  Hence it maps extreme points of $F_Z$ exactly to extreme points of $F_X$:
\[
    T_{U,x_0}(F_Z^\angle)=F_X^\angle .
\]
Because $F_Z$ is a face of $\mathcal Z_{U,x_0}$, $F_Z^\angle=F_Z\cap\mathcal Z_{U,x_0}^\angle$.  Because $F_X$ is a face of $\X$, $F_X^\angle=F_X\cap\X^\angle$.  Substituting these identities yields
\[
    T_{U,x_0}\!\left(F_{U,x_0}(c)\cap\mathcal Z_{U,x_0}^\angle\right)
    =
    \X^\star(c)\cap\X^\angle .
\]
\end{proof}

\begin{proof}[Verification of \Cref{ex:affine-vs-linear-projection}]
For every $c\in\C_\rho$, the inequality $c_1<0$ implies that $c_1x_1$ is uniquely minimized over $[-1,1]$ at $x_1=1$. Separability therefore gives
\[
  \X^\star(c)=\{1\}\times\operatorname*{argmin}_{z\in[-1,1]}c_2z.
\]
The second-coordinate optimizer set is $\{-1\}$ if $c_2>0$, $\{1\}$ if $c_2<0$, and $[-1,1]$ if $c_2=0$. Since $\C_\rho$ includes costs with $c_2=0$, the reachable optimizer set is the full right edge, and its decision-relevant subspace is $\spanop\{(0,1)\}$. The substitution $x=(1,z)$ maps the reduced optimizer set exactly onto $\X^\star(c)$ and gives the optimal value $c_1-|c_2|$. The reduction is minimal among affine restrictions preserving the full optimizer set throughout $\C_\rho$, because any such restriction must contain this nontrivial segment.
\end{proof}

\subsection{Proofs for Subsection~\ref{subsec:adj-learning}}
\label{app:proof-adj}

We first provide a geometric characterization of $\mathsf{Exact}(U,x_0,c)$.

\begin{lemma}[Characterization of $\mathsf{Exact}(U,x_0,c)$]
\label{lem:exact-slice-characterization}
Fix $x_0\in\X$ and a full-column-rank matrix $U$, and set
\[
    V_{U,x_0}:=x_0+\operatorname{range}(U).
\]
For any $c\in\RR^d$,
\[
    \mathsf{Exact}(U,x_0,c)
    \quad\Longleftrightarrow\quad
    \X^\star(c)\subseteq V_{U,x_0}.
\]
Consequently, if $U^+$ is obtained from $U$ by appending columns and
$\mathsf{Exact}(U,x_0,c)$ holds, then $\mathsf{Exact}(U^+,x_0,c)$ holds.
\end{lemma}

\begin{proof}
The map $T_{U,x_0}$ is an affine bijection from $\mathcal Z_{U,x_0}$ onto
$\X\cap V_{U,x_0}$. Hence
\[
    T_{U,x_0}\bigl(F_{U,x_0}(c)\bigr)
    =
    \argmin_{x\in\X\cap V_{U,x_0}} c^\top x .
\]
If $\mathsf{Exact}(U,x_0,c)$ holds, then the right-hand side is
$\X^\star(c)$, so $\X^\star(c)\subseteq V_{U,x_0}$.

Conversely, suppose $\X^\star(c)\subseteq V_{U,x_0}$. Since
$\X\cap V_{U,x_0}\subseteq\X$, the restricted LP cannot have value smaller
than the original LP. On the other hand, the original optimal face
$\X^\star(c)$ is feasible for the restricted LP by assumption. Thus the two
optimal values are equal, and the restricted optimizer set is
\[
    \argmin_{x\in\X\cap V_{U,x_0}} c^\top x
    =
    \X^\star(c)\cap V_{U,x_0}
    =
    \X^\star(c).
\]
This is exactly $\mathsf{Exact}(U,x_0,c)$.

Finally, appending columns enlarges $V_{U,x_0}$. Therefore containment of the
optimal face is preserved, and the monotonicity statement follows from the
equivalence just proved.
\end{proof}

To proceed, we show that the containment test and update in
Algorithm~\ref{alg:adj-cumulative} can be implemented by
solving LPs.

\begin{lemma}[Containment test via LPs]
\label{lem:adj-lp-containment}
Fix $c\in\RR^d$, an anchor $x_0\in\X$, and a matrix $U$. Let
$S:=\operatorname{range}(U)$, let
\[
    v(c):=\min_{y\in\X}c^\top y,
    \qquad
    G(c):=\X^\star(c)=\{x\in\X:c^\top x=v(c)\},
\]
and let $a_1,\ldots,a_{d-r}$ be any basis of $S^\perp$, where
$r:=\operatorname{rank}(U)$. Then
\[
    G(c)\subseteq x_0+S
\]
if and only if, for every $j=1,\ldots,d-r$,
\[
    \max_{x\in G(c)} a_j^\top(x-x_0)=0
    \quad\text{and}\quad
    \min_{x\in G(c)} a_j^\top(x-x_0)=0 .
\]
Moreover, if the containment fails, then one of these auxiliary LPs returns an
extreme optimizer $x\in G(c)\cap\X^\angle$ satisfying
$x-x_0\notin S$.
\end{lemma}

\begin{proof}
For any $x\in\RR^d$, we have $x\in x_0+S$ if and only if
$x-x_0\in S$, which is equivalent to
\[
    a_j^\top(x-x_0)=0
    \qquad
    \text{for all } j=1,\ldots,d-r.
\]
Therefore $G(c)\subseteq x_0+S$ holds if and only if each of the linear
functionals $a_j^\top(x-x_0)$ is identically zero on $G(c)$. This is equivalent
to requiring both its maximum and its minimum over $G(c)$ to be zero.

If the containment fails, then for some $j$ there exists
$\bar x\in G(c)$ with $a_j^\top(\bar x-x_0)\ne0$. Hence either
\[
    \max_{x\in G(c)}a_j^\top(x-x_0)>0
    \quad\text{or}\quad
    \min_{x\in G(c)}a_j^\top(x-x_0)<0 .
\]
Choose an extreme optimal solution $x$ of the corresponding auxiliary LP over
$G(c)$. Since $G(c)$ is a face of the polytope $\X$, every extreme point of
$G(c)$ is an extreme point of $\X$. The chosen point therefore belongs to
$G(c)\cap\X^\angle$ and satisfies $a_j^\top(x-x_0)\ne0$ for some $a_j\in S^\perp$,
so $x-x_0\notin S$.
\end{proof}

The next three lemmas formalize the key structural properties of our algorithm.
Together, they show that Algorithm~\ref{alg:adj-cumulative} induces a
\emph{stable, realizable} sample compression scheme
\citep[Definitions~7--8]{hanneke2021stablecompression} with compression size
at most $\dstar$. This is the key mechanism we use to obtain a
distribution-free fast-rate certificate \eqref{eq:fast-rate}.

\begin{lemma}[Compression size and termination]
\label{lem:adj-compression-size}
Fix an anchor $x_0\in\X^\star(\C)$ and run \Cref{alg:adj-cumulative}. Then the algorithm makes at most $\dstar$ total
appends. Moreover, for its output $(U_n,T_n)$,
\[
    \operatorname{range}(U_n)\subseteq\Wstar,
    \qquad
    |T_n|\le \operatorname{rank}(U_n)\le\dstar .
\]
\end{lemma}

\begin{proof}
Every appended column has the form $x-x_0$, where
$x\in\X^\star(c_i)\cap\X^\angle$ for some training sample $c_i\in\C$. Since
$x\in\X^\star(\C)$ and $x_0\in\X^\star(\C)$, we have
\[
    x-x_0\in \Wstar.
\]
Thus every appended direction lies in $\Wstar$.

The update rule appends such a vector only when
$x-x_0\notin\operatorname{range}(U_i)$. Hence each append increases the rank
of the current matrix by exactly one. Starting from the empty matrix, after
$k$ appends the current rank is $k$, and its range is a $k$-dimensional
subspace of $\Wstar$. Since $\dim(\Wstar)=\dstar$, there can be at most
$\dstar$ appends.

This also proves termination: if the while-loop were still active after the
current range had dimension $\dstar$, then the algorithm would find
$x\in\X^\star(c_i)\cap\X^\angle$ with
$x-x_0\notin\operatorname{range}(U_i)$. But all such differences belong to
$\Wstar$, and the current range is already a $\dstar$-dimensional subspace of
$\Wstar$, hence equals $\Wstar$, a contradiction.

Finally, every hard sample triggers at least one append, while each append
increases the rank by one. Therefore
\[
    |T_n|\le \operatorname{rank}(U_n)\le \dstar,
\]
and the range inclusion follows from the first paragraph.
\end{proof}

\begin{lemma}[Realizability]
\label{lem:adj-realizability}
Under the assumptions of \Cref{lem:adj-compression-size}, the output $U_n$ of
\Cref{alg:adj-cumulative} satisfies
\[
    \mathsf{Exact}(U_n,x_0,c_i)
    \qquad
    \text{for every } i=1,\ldots,n.
\]
Equivalently, the empirical failure loss
$1\{\neg\mathsf{Exact}(U_n,x_0,c_i)\}$ is zero on the training sample.
\end{lemma}

\begin{proof}
Fix a training index $i$. When the while-loop for $c_i$ terminates, it has
certified
\[
    \X^\star(c_i)\subseteq x_0+\operatorname{range}(U_i).
\]
By \Cref{lem:exact-slice-characterization}, this is equivalent to
$\mathsf{Exact}(U_i,x_0,c_i)$.

For later samples $t>i$, \Cref{alg:adj-cumulative} only appends columns to
the matrix. Therefore
\[
    \operatorname{range}(U_i)\subseteq\operatorname{range}(U_t)
    \qquad
    \text{for all } t\ge i.
\]
The monotonicity statement in \Cref{lem:exact-slice-characterization} then
implies $\mathsf{Exact}(U_n,x_0,c_i)$.
\end{proof}

\begin{lemma}[Stability]
\label{lem:adj-stability}
Let $S=(c_1,\ldots,c_n)$ be the training sequence and let
$(U_n,T_n)$ be the output of \Cref{alg:adj-cumulative}. Write
\[
    T_n=\{i_1<\cdots<i_k\},
    \qquad
    S_T:=(c_{i_1},\ldots,c_{i_k})
\]
where $S_T$ lists the hard samples in their original order.
Rerunning \Cref{alg:adj-cumulative} on $S_T$ with the same anchor $x_0$
returns $U_n$. More generally, deleting any subset of non-hard samples
from $S$ and rerunning the algorithm in the original sample order also
returns $U_n$.
\end{lemma}

\begin{proof}
For both algorithms, all choices follow fixed rules depending only on the current matrix, the fixed anchor, and the current cost. Let $U^{(j)}$ be the matrix after the $j$th hard sample $c_{i_j}$ in the original run, with $U^{(0)}$ empty. We show by induction that rerunning on $(c_{i_1},\ldots,c_{i_j})$ produces $U^{(j)}$.

The claim holds for $j=0$. Suppose it holds for $j-1$. Between the previous hard sample and $c_{i_j}$, all samples in the original run are non-hard and leave $U^{(j-1)}$ unchanged. Thus both runs process $c_{i_j}$ with the same matrix and anchor. They append the same directions and end with $U^{(j)}$.

Taking $j=k$ gives the same final matrix $U_n$. The same induction applies if any non-hard samples are retained: each enters with the original matrix and appends nothing. Hence deleting any subset of non-hard samples also leaves $U_n$ unchanged.
\end{proof}

We now provide the full proof of \Cref{thm:adj-cumulative-learning}.

\begin{proof}[Proof of \Cref{thm:adj-cumulative-learning}]
The structural claims
\[
    \operatorname{range}(U_n)\subseteq\Wstar,
    \qquad
    |T_n|\le \operatorname{rank}(U_n)\le\dstar
\]
are exactly \Cref{lem:adj-compression-size}. The zero-training-failure claim
is \Cref{lem:adj-realizability}. It remains to prove the out-of-sample
certificate \eqref{eq:fast-rate}.

\paragraph{Sample-compression background.}
Sample compression is a classical way to prove distribution-free generalization:
rather than controlling the complexity of an entire hypothesis class, one shows
that the learned predictor can be reconstructed from a small subset of the
training sample.  This idea goes back to the PAC-learning and compression
literature \citep{valiant1984learnable,littlestone1986relating,
floyd1995samplecompression}; see also
\citet{graepel2005pacbayes,moran2016samplecompression} for later developments.
The particular form needed here is the \emph{stable, realizable} version of
sample compression.  Stable compression schemes and their fast-rate consequences
are developed in \citet{bousquet2020properlearning} and
\citet{hanneke2021stablecompression}; related compression-based certificates in
optimization and scenario analysis are discussed by
\citet{campigaratti2023compression}.

We recall the specific form used below.  A sample compression scheme for binary
prediction consists of a compression map $\kappa$, which selects an ordered
subsequence of the training sample, and a reconstruction map $\rho$, which maps
the compressed subsequence to a prediction rule.  The scheme is \emph{stable} if
deleting any uncompressed training examples does not change the reconstructed
rule.  In the realizable stable case, where the reconstructed rule has zero
empirical error, \citet[Definitions~7--8 and Corollary~11]{hanneke2021stablecompression}
imply that, with probability at least $1-\delta$ over an i.i.d. sample $S$ of
size $n$,
\[
    R(\rho(\kappa(S)))
    \le
    \frac{4}{n}\left(6|\kappa(S)|+\log\frac{e}{\delta}\right),
\]
where $R(\cdot)$ is the true $0$--$1$ risk and $|\kappa(S)|$ is the realized
compression size.  This fast $1/n$ rate is specific to the realizable stable
regime; in agnostic compression settings, such fast rates are not available in
general \citep{hanneke19b}.

\paragraph{The induced binary learning problem.}
For a matrix $U$, define the binary failure rule
\[
    h_U(c):=1\{\neg\mathsf{Exact}(U,x_0,c)\}.
\]
Equivalently, we may view the data as a supervised binary problem with examples
$(c,Y)$, where $c\sim P_c$ and $Y=0$ almost surely. Under this convention,
the risk of $h_U$ is exactly the exact-compression failure probability:
\[
    R(h_U)
    =
    \mathbb P_{c\sim P_c}\!\left[h_U(c)=1\right]
    =
    \mathbb P_{c\sim P_c}\!\left[\neg\mathsf{Exact}(U,x_0,c)\right].
\]

Let $S=(c_1,\ldots,c_n)$ and let $(U_n,T_n)$ be the output of
\Cref{alg:adj-cumulative}. Define the compression map $\kappa$ by
\[
    \kappa(S):=(c_i)_{i\in T_n},
\]
where the subsequence is stored in its original order. Since all labels are
identically zero, we suppress them in the notation; equivalently, $\kappa$
stores the labeled examples $(c_i,0)_{i\in T_n}$. Define the reconstruction map
$\rho$ by rerunning \Cref{alg:adj-cumulative} on any supplied subsequence
$S'$ using the same fixed anchor $x_0$, and
returning the classifier $h_{U(S')}$ induced by the resulting matrix $U(S')$.

By \Cref{lem:adj-stability}, rerunning the algorithm on the compressed hard
subsequence $\kappa(S)$ returns the same matrix $U_n$. Therefore
\[
    \rho(\kappa(S))=h_{U_n}.
\]
The same lemma also gives the stability property required by
\citet{hanneke2021stablecompression}: deleting any subset of the uncompressed
samples does not change the reconstructed classifier. By
\Cref{lem:adj-realizability},
\[
    h_{U_n}(c_i)=0
    \qquad
    \text{for every } i=1,\ldots,n,
\]
so the reconstructed classifier has zero empirical error.

Applying \citet[Corollary~11]{hanneke2021stablecompression} to this stable
realizable compression scheme yields, with probability at least $1-\delta$,
\[
\begin{aligned}
    \mathbb P_{c\sim P_c}\!\left[\neg\mathsf{Exact}(U_n,x_0,c)\right]
    &=
    R(h_{U_n}) \\
    &=
    R(\rho(\kappa(S))) \\
    &\le
    \frac{4}{n}\left(6|\kappa(S)|+\log\frac{e}{\delta}\right) \\
    &=
    \frac{4}{n}\left(6|T_n|+\log\frac{e}{\delta}\right).
\end{aligned}
\]
Taking complements gives
\[
    \mathbb P_{c\sim P_c}\!\left[\mathsf{Exact}(U_n,x_0,c)\right]
    \ge
    1-\frac{4}{n}\left(6|T_n|+\log\frac{e}{\delta}\right).
\]
The final inequality in the theorem follows from
$|T_n|\le\dstar$ in \Cref{lem:adj-compression-size}.

\paragraph{Anchor generated from data.}
The theorem assumes that $x_0$ is fixed independently of the $n$ training
costs. If instead $x_0$ is selected from $\X^\star(c_1)$ by a fixed deterministic rule, retain $c_1$ in the compressed subsequence so that reconstruction can first recover the same anchor. The bound is then unchanged
except that $|T_n|$ is replaced by $|T_n|+1$. Equivalently, one may draw an
independent anchor sample first and apply the theorem to the remaining
representation-learning samples.

\paragraph{LP implementation and tractability.}
It remains to justify the claimed LP-based implementation. By
\Cref{lem:adj-lp-containment}, each containment test
\[
    \X^\star(c_i)\subseteq x_0+\operatorname{range}(U_i)
\]
can be implemented by first solving the LP defining $v(c_i)$ and then solving,
for a basis $a_1,\ldots,a_{d-r}$ of $\operatorname{range}(U_i)^\perp$, the
auxiliary LPs
\[
    \max_{x\in\X:\ c_i^\top x=v(c_i)} a_j^\top(x-x_0),
    \qquad
    \min_{x\in\X:\ c_i^\top x=v(c_i)} a_j^\top(x-x_0).
\]
If all these optima are zero, containment is certified. If not, an extreme
optimal solution of a violated auxiliary LP is a valid vertex
$x\in\X^\star(c_i)\cap\X^\angle$ with
$x-x_0\notin\operatorname{range}(U_i)$, so $x-x_0$ is a valid update direction in
\Cref{alg:adj-cumulative}.

There is one final containment check for each of the $n$ samples and one
additional check for each append. Since \Cref{lem:adj-compression-size} gives at
most $\dstar$ total appends, the number of containment checks is at most
$n+\dstar$. Each check solves only a polynomial number of LPs over $\X$ or over
the optimal face of $\X$ obtained by adding the equality $c_i^\top x=v(c_i)$,
and the required orthogonal bases and rank tests are computable by standard
linear algebra. Thus, when $\X$ is given in the H-representation of
\eqref{eq:lp}, the cumulative learner is polynomial time in the usual LP oracle
model.
\end{proof}

\subsection{Proofs for Subsection~\ref{subsec:prior-estimation}}

\subsubsection{Construction of calibrated estimated priors}
\label{app:estimated-prior-construction}
We construct a distribution-free calibrated prior using the one-sided tolerance-region argument of \citet{wilks1941determination}. The construction allows arbitrary data-fitted scores and is closely related to split conformal calibration \citep{vovk2005algorithmic,vovk2012conditional,lei2018distribution}. Split an independent pilot sample into a fitting sample $\mathcal D_{\rm fit}$ and a calibration sample $\mathcal D_{\rm cal}=\{c^{\rm cal}_1,\ldots,c^{\rm cal}_m\}$.  Fit a measurable score $s_{\widehat\theta}:\RR^d\to\RR$ using only $\mathcal D_{\rm fit}$, and assume every sublevel set $\{c:s_{\widehat\theta}(c)\le t\}$ is convex.  Let $S_j=s_{\widehat\theta}(c^{\rm cal}_j)$, write $S_{(1)}\le\cdots\le S_{(m)}$ for the order statistics, and set $S_{(m+1)}:=+\infty$.

Fix $\rho,\delta_0\in(0,1)$.  Define
\begingroup{}
\[
\begin{aligned}
    k_{\rho,\delta_0}&:=
    \min\left\{k\in\{1,\ldots,m+1\}:
        \mathbb P[\operatorname{Bin}(m,1-\rho)\ge k]\le\delta_0
    \right\},
    \\
    \widehat\C_{\rho,\delta_0}&:=\{c:s_{\widehat\theta}(c)\le S_{(k_{\rho,\delta_0})}\}.
\end{aligned}
\]
\endgroup{}
\begingroup{}
The number of calibration scores strictly below a population $(1-\rho)$-quantile is stochastically dominated by $\operatorname{Bin}(m,1-\rho)$. \Cref{lem:binomial-cutoff-size} bounds the resulting tail cutoff. If the calibration sample is too small for the requested confidence, $k_{\rho,\delta_0}=m+1$ and the calibrated set is all of $\RR^d$.
\endgroup{}

\begin{proposition}[Calibrated convex estimated prior]
\label{prop:tolerance-estimated-prior}
Under the preceding calibration setup, $\widehat\C_{\rho,\delta_0}$ is convex and is a $(\rho,\delta_0)$-estimated prior.
\end{proposition}

A simple score with convex sublevel sets is the ridge Mahalanobis score $s_{\widehat\theta}(c):=(c-\widehat\mu)^\top(\widehat\Sigma+\lambda I)^{-1}(c-\widehat\mu)$, where $\lambda>0$ and $\widehat\mu,\widehat\Sigma$ are fitted on $\mathcal D_{\rm fit}$.  This is the squared Mahalanobis distance \citep{mahalanobis1936generalized} computed with a ridge covariance estimate; the ridge term makes the inverse well-defined and stabilizes high-dimensional covariance estimation \citep{ledoit2004wellconditioned}.  \begingroup{}When the cutoff is finite, the calibrated estimated prior is an ellipsoid.\endgroup{}

\begin{proof}[Proof of \Cref{prop:tolerance-estimated-prior}]
Condition on the fitting sample $\mathcal D_{\rm fit}$, so the score $s_{\widehat\theta}$ is fixed.  Let $S=s_{\widehat\theta}(c)$ for an independent fresh draw $c\sim P_c$, and write $F(t):=\mathbb P[S\le t]$.  Let
\[
    q_{1-\rho}:=\inf\{t\in\RR:F(t)\ge 1-\rho\}
\]
be a $(1-\rho)$-quantile of the score distribution.  By right-continuity of $F$, $F(q_{1-\rho})\ge 1-\rho$, and $F(q_{1-\rho}^-):=\mathbb P[S<q_{1-\rho}]\le 1-\rho$.

If $F(S_{(k)})<1-\rho$, then $S_{(k)}<q_{1-\rho}$.  Hence at least $k$ of the calibration scores are strictly smaller than $q_{1-\rho}$.  The number of such scores has distribution $\operatorname{Bin}(m,F(q_{1-\rho}^-))$ and is stochastically dominated by $\operatorname{Bin}(m,1-\rho)$.  Therefore, for $k=k_{\rho,\delta_0}$,
\[
\begin{aligned}
    \mathbb P_{\rm cal}\bigl[F(S_{(k)})<1-\rho\mid \mathcal D_{\rm fit}\bigr]
    &\le
    \mathbb P\bigl[\operatorname{Bin}(m,F(q_{1-\rho}^-))\ge k\bigr] \\
    &\le
    \mathbb P\bigl[\operatorname{Bin}(m,1-\rho)\ge k\bigr]
    \le \delta_0 .
\end{aligned}
\]
Equivalently, with conditional probability at least $1-\delta_0$, $P_c(c\in\widehat\C_{\rho,\delta_0})=F(S_{(k)})\ge 1-\rho$.  The same conclusion holds after removing the conditioning on $\mathcal D_{\rm fit}$.  Convexity follows because $\widehat\C_{\rho,\delta_0}$ is a sublevel set of $s_{\widehat\theta}$; if $k=m+1$, then $S_{(m+1)}=+\infty$ and the set is all of $\RR^d$.
\end{proof}

\begingroup{}
\begin{lemma}[Size of the cutoff]
\label{lem:binomial-cutoff-size}
Let $X\sim\operatorname{Bin}(m,1-\rho)$ and $L:=\log(1/\delta_0)$.  Then
\[
    k_{\rho,\delta_0}
    \le
    \min\left\{m+1,\left\lfloor m(1-\rho)+\sqrt{2m\rho(1-\rho)L}+\frac{2L}{3}\right\rfloor+1\right\}.
\]
\end{lemma}
\endgroup{}

\begin{proof}
The Bernstein--Chernoff bound for a binomial random variable gives, for every $t\ge0$,
\[
    \mathbb P[X-m(1-\rho)\ge t]
    \le
    \exp\!\left(-\frac{t^2}{2(m\rho(1-\rho)+t/3)}\right).
\]
With $t=\sqrt{2m\rho(1-\rho)L}+2L/3$, $\mathbb P[X\ge m(1-\rho)+t]\le e^{-L}=\delta_0$. Therefore any integer strictly larger than $m(1-\rho)+t$ is feasible in the definition of $k_{\rho,\delta_0}$; the displayed bound follows after truncating at $m+1$.
\end{proof}

\subsubsection{Total and retained sample sizes}
\label{app:retained-sample-size}
Here $n$ is the total number of training costs before filtering, separate from the independent pilot used to construct $\widehat\C$. Conditional on the pilot, write $p:=P_c(c\in\widehat\C)$. The retention indicators are independent Bernoulli variables with success probability $p$, so
\[
    \hat n\mid\widehat\C\sim\operatorname{Bin}(n,p).
\]
For any $\eta\in(0,1)$, Hoeffding's inequality gives
\[
    \mathbb P\!\left[
        \left|\frac{\hat n}{n}-p\right|>
        \sqrt{\frac{\log(2/\eta)}{2n}}
        \,\middle|\,\widehat\C
    \right]\le\eta.
\]
Thus $\hat n\approx np$, and $\hat n\approx(1-\rho)n$ when $p$ is close to $1-\rho$. On $\mathcal E_0$, we have $p\ge1-\rho$, which gives
\[
    \hat n\ge (1-\rho)n-\sqrt{\frac{n\log(2/\eta)}{2}}
\]
with joint probability at least $1-\delta_0-\eta$. For fixed $\rho<1$, the $1/\hat n$ term in the learning bound therefore retains the $1/n$ rate as $n$ grows.

\begin{proof}[Proof of \Cref{thm:calibrated-prior-learning}]
By the $(\rho,\delta_0)$-estimated-prior property, $\mathbb P_{\rm pilot}(\mathcal E_0)\ge1-\delta_0$.  Condition on a pilot realization in $\mathcal E_0$.  Then $P_c(c\in\widehat\C)\ge1-\rho>0$, so $P_{\widehat\C}=P_c(\cdot\mid c\in\widehat\C)$ is well-defined and supported on the convex prior $\widehat\C$.  The chosen anchor $\widehat x_0$ belongs to $\X^\star(\widehat\C)$ by construction.

Conditional also on $\hat n=k\ge1$, the retained costs, in their original order, are i.i.d. from $P_{\widehat\C}$. Applying \Cref{thm:adj-cumulative-learning} to these $k$ costs gives, with conditional probability at least $1-\delta_1$,
\[
    \mathbb P_{c\sim P_{\widehat\C}}\!\left[
        \neg\mathsf{Exact}(U_{\hat n},\widehat x_0,c)
    \right]
    \le
    \frac{4}{\hat n}\bigl(6\widehat d_\star+\log(e/\delta_1)\bigr).
\]
On this event,
\[
\begin{aligned}
    \mathbb P_{c\sim P_c}\!\left[\neg\mathsf{Exact}(U_{\hat n},\widehat x_0,c)\right]
    &\le P_c(c\notin\widehat\C)
    +P_c(c\in\widehat\C)\,
      \mathbb P_{c\sim P_{\widehat\C}}\!\left[
        \neg\mathsf{Exact}(U_{\hat n},\widehat x_0,c)
      \right] \\
    &\le \rho+
    \frac{4}{\hat n}\bigl(6\widehat d_\star+\log(e/\delta_1)\bigr).
\end{aligned}
\]
Taking complements gives \eqref{eq:calibrated-prior-risk}. Averaging the conditional failure probability over $k\ge1$, and adding the probability of $\mathcal E_0^c$, gives a joint failure probability of at most $\delta_0+\delta_1$. The theorem makes no assertion when $\hat n=0$. The final statement follows by applying \Cref{prop:tolerance-estimated-prior}, which supplies a $(\rho,\delta_0)$-estimated prior.
\end{proof}

\section{Compression for Integer Programming}
\label{app:ip-extension}

In this section, we present our compression framework for integer programming. We first characterize exact compression for IP (Sec.~\ref{app:ip-exact}), then give a learning algorithm (Sec.~\ref{app:ip-learning}) and evaluate it on unit commitment and knapsack problems (Sec.~\ref{app:ip-experiment}).

Let
\[
    \X_I=\{x\in\mathbb Z^d:Ax\le b\}
\]
be nonempty and finite, and let $\X_I^\star(c)=\argmin_{x\in\X_I}c^\top x$. Define
\[
    \X_I^\star(\C)=\bigcup_{c\in\C}\X_I^\star(c),
    \qquad
    W_I^\star=\spanop\{x-y:x,y\in\X_I^\star(\C)\},
    \qquad d_I^\star=\dim W_I^\star.
\]
Choose a reachable anchor $x_0\in\X_I^\star(\C)$. For a matrix $U\in\RR^{d\times r}$ of full column rank, the restricted problem is
\begin{equation}
\label{eq:ip-affine-restriction}
    \min_{z\in\RR^r}(U^\top c)^\top z
    \quad\text{subject to}\quad
    A(x_0+Uz)\le b,\qquad x_0+Uz\in\mathbb Z^d.
\end{equation}
The original objective includes the additional constant $c^\top x_0$. Equivalently, if the columns of $N$ span $\operatorname{range}(U)^\perp$, one can keep the original integer variables and add only
\begin{equation}
\label{eq:ip-original-variable-slice}
    N^\top(x-x_0)=0.
\end{equation}
This form can be passed directly to an integer programming solver.

\subsection{The Geometry of Exact Compression for IP}
\label{app:ip-exact}

We first present a geometric characterization of IP compression. The following proposition specializes \Cref{thm:exact-reduced-lp} to integer feasible sets.

\begin{proposition}
\label{prop:ip-affine-exactness}
For any affine space $V$, the identity
\[
    \argmin_{x\in\X_I\cap V}c^\top x=\X_I^\star(c)
    \quad\text{for every }c\in\C
\]
holds if and only if $\X_I^\star(\C)\subseteq V$. Consequently, the minimum dimension of an affine restriction preserving every optimizer for every $c\in\C$ is $d_I^\star$, attained by $V=x_0+W_I^\star$.
\end{proposition}
\begin{proof}
We follow the proof of \Cref{thm:exact-reduced-lp} in \Cref{app:proof-exact}, with $\X_I^\star(c)$ in place of $\X^\star(c)$. Suppose first that the optimizer sets are equal for every $c\in\C$. Then $\X_I^\star(c)\subseteq V$ for every $c\in\C$, so $\X_I^\star(\C)\subseteq V$. Conversely, suppose $\X_I^\star(\C)\subseteq V$. For each $c\in\C$, the restricted feasible set contains an original optimizer, so
\[
    \min_{x\in\X_I\cap V}c^\top x
    =\min_{x\in\X_I}c^\top x=:v(c).
\]
It follows that
\[
    \argmin_{x\in\X_I\cap V}c^\top x
    =\{x\in\X_I\cap V:c^\top x=v(c)\}
    =\X_I^\star(c).
\]
It remains to identify the minimum dimension. Any affine space $V$ containing $\X_I^\star(\C)$ also contains the difference $x-y$ in its direction space for every $x,y\in\X_I^\star(\C)$. Taking spans gives
\[
    W_I^\star\subseteq\dir(V),
    \qquad \dim V\ge d_I^\star.
\]
Finally, $x_0\in\X_I^\star(\C)$ gives
\[
    \X_I^\star(\C)\subseteq x_0+W_I^\star,
    \qquad \dim(x_0+W_I^\star)=d_I^\star,
\]
so this affine space attains the minimum dimension.
\end{proof}

\paragraph{Integer coordinates for the compressed formulation.}\leavevmode
In \eqref{eq:ip-affine-restriction}, $z$ has real coordinates, while $x_0+Uz$ must be integral. Even when $U$ has integer entries, replacing $x_0+Uz\in\mathbb Z^d$ by $z\in\mathbb Z^r$ can remove feasible integer solutions. For example, $U=[2]$ with $x_0=0$ has range $\RR$, but restricting $z$ to integers excludes $x=1$.

For a subspace learned by \Cref{alg:ip-cumulative}, we can express the same restriction with $r$ integer variables. The columns of $U$ are integer differences, so the lattice $\operatorname{range}(U)\cap\mathbb Z^d$ has rank $r$. Choose an integer basis $B$ of this lattice. Then
\[
    (x_0+\operatorname{range}(U))\cap\mathbb Z^d=x_0+B\mathbb Z^r,
\]
and the equivalent formulation has $z\in\mathbb Z^r$ and constraints $A(x_0+Bz)\le b$. The original difference columns need not form this lattice basis.

\subsection{Learning Exact Compression}
\label{app:ip-learning}

We next extend the sample-based algorithm of \Cref{subsec:adj-learning} to IPs. The learner adds differences from a fixed anchor to integer optimizers until the learned affine space contains all optimizers of each training instance. As in \Cref{thm:adj-cumulative-learning}, we bound the probability of failing to preserve the full optimizer set for an independent test cost, with high confidence over the training sample.

Fix $x_0\in\X_I^\star(\C)$ independently of the training costs. \Cref{alg:ip-cumulative} starts with an empty matrix and processes $c_1,\ldots,c_n$ one at a time. It appends $x-x_0$ whenever an optimizer $x\in\X_I^\star(c_i)$ is not in the current affine space. Samples that trigger at least one append are called hard instances, and their indices form $T_n$.

\begin{algorithm}[htbp]
\caption{Learning exact IP compression}
\label{alg:ip-cumulative}
\begin{algorithmic}[1]
\Require Finite integer feasible set $\X_I$, fixed reachable anchor $x_0\in\X_I^\star(\C)$, training samples $c_1,\ldots,c_n$, and deterministic tie-breaking for all choices.
\Ensure A compression matrix $U_n$ and a hard-sample index set $T_n$.
\State Initialize $U_0\gets [\,]\in\RR^{d\times 0}$ and $T_n\gets\emptyset$.
\For{$i=1,\ldots,n$}
    \State Set $U_i\gets U_{i-1}$.
    \While{$\X_I^\star(c_i)\nsubseteq x_0+\operatorname{range}(U_i)$}
        \State Choose $x\in\X_I^\star(c_i)$ with $x-x_0\notin\operatorname{range}(U_i)$.
        \State $U_i\gets [\,U_i\;\;x-x_0\,]$ and $T_n\gets T_n\cup\{i\}$.
    \EndWhile
\EndFor
\State \Return $U_n$ and $T_n$.
\end{algorithmic}
\end{algorithm}

Compared with \Cref{alg:adj-cumulative}, \Cref{alg:ip-cumulative} uses integer optimizers $x\in\X_I^\star(c_i)$ in place of LP vertex optimizers and tests containment by solving IPs. For the complexity bound below, let $U$ be the current compression matrix, $r=\operatorname{rank}(U)$, and $v_i=\min_{x\in\X_I}c_i^\top x$. The next theorem gives the IP counterpart of \Cref{thm:adj-cumulative-learning}. Its proof implements the containment test using the orthogonal directions in \Cref{lem:adj-lp-containment}.

\begin{theorem}
\label{thm:ip-sample-learning}
For $n\ge1$ i.i.d. costs from a distribution $P_c$ supported on $\C$, this learner satisfies
\[
    |T_n|\le\operatorname{rank}(U_n)\le d_I^\star,
    \qquad
    \X_I^\star(c_i)\subseteq x_0+\operatorname{range}(U_n)
    \quad (i=1,\ldots,n).
\]
For any $\delta\in(0,1)$, with probability at least $1-\delta$ over the training sample,
\begin{equation}
\label{eq:ip-sample-risk}
    \mathbb P_{c\sim P_c}\!\left[
        \X_I^\star(c)\nsubseteq x_0+\operatorname{range}(U_n)
    \right]
    \le \frac4n\left(6|T_n|+\log\frac e\delta\right).
\end{equation}
There are at most $d_I^\star$ appends and $n+d_I^\star$ containment tests. Each test requires at most one IP for $v_i$ and $2(d-r)$ auxiliary IPs.
\end{theorem}
\begin{proof}
We follow the proof of \Cref{thm:adj-cumulative-learning}, replacing LP vertex optimizers by integer optimizers.

\noindent{\textit{(i) Compression size and training exactness.}}
As in \Cref{lem:adj-compression-size}, each update in \Cref{alg:ip-cumulative} satisfies
\[
    x-x_0\in W_I^\star\setminus\operatorname{range}(U),
    \qquad
    \operatorname{rank}[\,U\;x-x_0\,]=\operatorname{rank}(U)+1.
\]
Every appended column therefore lies in $W_I^\star$, and each append increases the rank by one. There are at most $d_I^\star$ appends. Since each hard instance causes at least one append, $|T_n|\le\operatorname{rank}(U_n)\le d_I^\star$. The same monotonicity as in \Cref{lem:adj-realizability} gives
\[
    \X_I^\star(c_i)\subseteq x_0+\operatorname{range}(U_i)
    \subseteq x_0+\operatorname{range}(U_n).
\]
The first inclusion holds when the while-loop for $c_i$ ends, and the second holds because subsequent updates only append columns. Since every appended vector comes from a training optimizer, these inclusions also give
\[
    x_0+\operatorname{range}(U_n)
    =\operatorname{aff}\left(\{x_0\}\cup\bigcup_{i=1}^n\X_I^\star(c_i)\right).
\]

\noindent{\textit{(ii) Stable reconstruction.}}
For $S=(c_1,\ldots,c_n)$, write $U(S)$ for the output matrix and $S_T=(c_i)_{i\in T_n}$ in its original order. If $j\notin T_n$, processing $c_j$ leaves the matrix unchanged. Deleting $c_j$ therefore leaves the input matrix at the next sample unchanged, and induction gives the same matrix at every later sample. This is exactly the induction in the proof of \Cref{lem:adj-stability}, and it gives
\[
    U(S_T)=U(S),\qquad
    U(S\setminus J)=U(S)\quad\text{for every }J\subseteq\{1,\ldots,n\}\setminus T_n,
\]
where deletion preserves the order of the remaining costs.

\noindent{\textit{(iii) Failure probability.}}
As in the proof of \Cref{thm:adj-cumulative-learning}, assign the label zero to every cost and define
\[
    h_U(c)=1\{\X_I^\star(c)\nsubseteq x_0+\operatorname{range}(U)\}
\]
The compression map retains $S_T$, and the reconstruction map reruns \Cref{alg:ip-cumulative} on that subsequence with the same anchor. The preceding identities show that deleting unretained costs leaves the reconstructed rule $h_{U_n}$ unchanged. The training containment proved above gives $h_{U_n}(c_i)=0$ for all $i$. Conditional on the independent anchor, \citet[Corollary~11]{hanneke2021stablecompression} therefore yields \eqref{eq:ip-sample-risk} with compression size $|T_n|$.

\noindent{\textit{(iv) Containment tests.}}
We now implement the containment test, following \Cref{lem:adj-lp-containment}. Let $v_i=\min_{x\in\X_I}c_i^\top x$, let $r=\operatorname{rank}(U)$, and choose a basis $a_1,\ldots,a_{d-r}$ of $\operatorname{range}(U)^\perp$. Then
\[
\begin{aligned}
    &\X_I^\star(c_i)\subseteq x_0+\operatorname{range}(U)\\
    &\qquad\Longleftrightarrow\quad
    \min_{x\in\X_I^\star(c_i)}a_j^\top(x-x_0)
    =\max_{x\in\X_I^\star(c_i)}a_j^\top(x-x_0)=0
\end{aligned}
\]
for every $j=1,\ldots,d-r$. Indeed, these linear functionals all vanish at $x$ exactly when $x-x_0\in\operatorname{range}(U)$. Each extremum is attained because $\X_I^\star(c_i)$ is finite and nonempty. A nonzero extremum gives $x\in\X_I^\star(c_i)$ with $x-x_0\notin\operatorname{range}(U)$. The test uses at most one IP to compute $v_i$ and $2(d-r)$ IPs with the added constraint $c_i^\top x=v_i$. The value $v_i$ can be reused while processing the same cost. There is one successful test per sample, and each failed test adds one column. The number of tests is therefore at most $n+d_I^\star$.
\end{proof}

We next discuss four points about implementing and extending the learning algorithm.

\begin{enumerate}[label=\arabic*.,ref=\arabic*,leftmargin=1.4em,itemsep=4pt,topsep=4pt,parsep=3pt]
\item \textbf{Choosing the anchor.}
As in the proof of \Cref{thm:adj-cumulative-learning}, we can also choose the anchor from the first training instance. If $x_0$ is selected from $\X_I^\star(c_1)$ by a fixed deterministic rule, retain $c_1$ so that reconstruction can first recover the same anchor. The compression size is then $|T_n\cup\{1\}|\le|T_n|+1$, so \eqref{eq:ip-sample-risk} remains valid with $|T_n|+1$ in place of $|T_n|$.

\item \textbf{Parallel implementation.}
When optimizers are unique almost surely, we can use the parallel implementation from \Cref{subsec:adj-learning}. Let $x_i$ be the optimizer for $c_i$. The affine-hull identity in the proof of \Cref{thm:ip-sample-learning} becomes
\[
    x_0+\operatorname{range}(U_n)
    =x_0+\spanop\{x_i-x_0:i=1,\ldots,n\}.
\]
We can therefore solve the $n$ training IPs in parallel and then compute a basis of these differences. Absolute continuity of $P_c$ on $\RR^d$ suffices for uniqueness because ties between two distinct integer points require $c^\top(x-y)=0$, and $\X_I$ is finite.

If an objective has several optimizers, choose one by a fixed rule depending only on the cost, and write $x(c)$ for this choice. The proof of \Cref{thm:ip-sample-learning}, applied to the failure rule $1\{x(c)\notin x_0+\operatorname{range}(U)\}$, gives the bound in \eqref{eq:ip-sample-risk} with $T_n$ defined by updates from the selected optimizers. The restricted problem loses the optimal value exactly when
\[
    \X_I^\star(c)\cap(x_0+\operatorname{range}(U_n))=\varnothing.
\]
This event implies $x(c)\notin x_0+\operatorname{range}(U_n)$, so the same probability bound applies. Preserving the entire optimizer set instead uses the containment tests in \Cref{alg:ip-cumulative}.

\item \label{item:ip-estimated-prior}\textbf{Learning with an estimated prior.}
The estimated-prior construction in \Cref{thm:calibrated-prior-learning} also applies to IPs. Let $\widehat\C$ be an independently calibrated estimated prior whose excluded mass is at most $\rho$ with confidence $1-\delta_0$. Draw $n\ge1$ i.i.d. costs independently of the pilot and retain the $\hat n$ costs inside $\widehat\C$. Choose an anchor $\widehat x_0\in\X_I^\star(\widehat\C)$ independently of these costs. Conditional on the pilot and $\hat n=k\ge1$, apply \Cref{thm:ip-sample-learning} to the retained sample from $P_c(\cdot\mid c\in\widehat\C)$. The argument in Appendix~\ref{app:retained-sample-size} gives, with probability at least $1-\delta_0-\delta_1$, the following full-distribution failure bound whenever $\hat n\ge1$:
\[
    \rho+\frac4{\hat n}\left(6\widehat d_I^\star+\log\frac e{\delta_1}\right),
\]
where $\widehat d_I^\star$ is the dimension of the span of differences between optimizers reachable under $\widehat\C$.

\item \label{item:ip-mixed-integer}\textbf{Extension to mixed integer programs.}
We can also restrict a selected block of integer variables in a mixed integer program. Consider a nonempty bounded mixed integer feasible set in variables $(x,y)$, where $x\in\mathbb Z^p$ is the selected block and $y$ contains the remaining integer and continuous variables. For each cost, project the joint optimizer set onto the $x$ coordinates and apply \Cref{alg:ip-cumulative} to these finite sets. The learned restriction is $x\in x_0+\operatorname{range}(U)$, with all original variable domains and constraints retained. It preserves every joint optimizer exactly when it contains every optimal $x$ vector. The proofs of \Cref{prop:ip-affine-exactness,thm:ip-sample-learning} therefore apply with $\X_I^\star(c)$ replaced by the projected joint optimizer set. The containment tests become mixed integer programs over the joint optimal set. Sampling one optimizer per cost suffices for full-set preservation when the optimal $x$ vector is unique almost surely.
\end{enumerate}

\subsection{Numerical Experiments}
\label{app:ip-experiment}
We evaluate our algorithm on unit commitment (UC) and knapsack problems with several baselines.

\subsubsection{Problem formulations.}
\paragraph{Unit commitment.}
Unit commitment decides which generators run and how much power they produce to meet demand at minimum cost \citep{xavier2021learning}. For generator $g$ in period $t$, binary variables $u_{gt}$, $v_{gt}$, and $w_{gt}$ indicate whether it is on, turns on, or turns off. The continuous variable $p_{gt}\ge0$ is its power output. We minimize
\[
 \sum_{g,t}\bigl(c_g^u u_{gt}+c_g^p p_{gt}+c_g^v v_{gt}\bigr)
 +\text{penalties}
 \quad\text{over the fixed feasible set }\mathcal F_{\rm UC}.
\]
The fixed feasible set $\mathcal F_{\rm UC}$ imposes power balance, generator output limits, ramp limits, minimum on/off durations, and reserve requirements. We retain the benchmark's network constraints and penalty variables.

In our experiments, we use the public \texttt{matpower/case14/2017-02-01} instance from UnitCommitment.jl v0.4.3 \citep{xavier2024unitcommitment}. It has five generators, 14 buses, 20 lines, and 36 hourly periods. We keep its physical data fixed and sample generator costs using the ranges in the MIPLearn v0.4 UC generator.\footnote{\url{https://anl-ceeesa.github.io/MIPLearn/0.4/guide/problems/\#unit-commitment}.} Independently for each generator and cost sample, we draw
\[
 c_g^u\sim\operatorname{Unif}[0,1000],\qquad
 c_g^p\sim\operatorname{Unif}[0,50],\qquad
 c_g^v\sim\operatorname{Unif}[0,10000].
\]
These coefficients are rounded to two decimal places and remain constant over the horizon. The benchmark's penalty costs remain fixed. All methods use the same model after exact elimination of redundant auxiliary variables, with 2,844 variables, 2,544 constraints, and 540 binary variables. We compress its 180 on/off variables.

\paragraph{Knapsack.}
Knapsack selects items of maximum total value subject to capacity constraints \citep{kellerer2004knapsack},
\[
 \max_{x\in\{0,1\}^d}q^\top x
 \quad\text{subject to}\quad \Omega x\le B.
\]
Here $q$ contains item values, and $\Omega x\le B$ describes capacities and, when present, conflicts between items. This has the minimization form above with $c=-q$. We consider two synthetic settings.
\begin{enumerate}
\item \textbf{Grouped knapsack.} The first knapsack model has 32 groups, each containing two basic items and one upgrade, giving 96 binary variables. Each group has its own capacity, allowing both basic items together or the upgrade alone, but preventing the upgrade from being combined with either basic item. A fixed conflict graph specifies pairs of upgrades that cannot be selected together \citep{pferschy2009knapsack}. These constraints remain fixed across all training and test samples. Each basic item has value $0.25$. For each cost sample, the 32 upgrade values are drawn independently from the uniform distribution on $[0.1,2.5]$.

\item \textbf{Correlated knapsack.} We also use correlated knapsack instances based on \citet{pisinger2005hard}. Each instance has 100 items with weights $a_j$ drawn independently and uniformly from the integers $1,\ldots,1000$. Its single capacity is $\lfloor\frac12\sum_j a_j\rfloor$. After the weights and capacity are fixed, the item values (cost vectors) are independently sampled for each query as
\[
 q_j=a_j+100+\varepsilon_j,
 \qquad \varepsilon_j\sim\operatorname{Unif}\{-2,-1,0,1,2\}.
\]

\end{enumerate}

\subsubsection{Experimental setup.}
\paragraph{Training and testing scale.} 
Each sample is a draw of the objective coefficient vector $c$. UC experiments use 512 training samples, 64 separate validation samples, and 256 independent test samples. Grouped knapsack uses three instances, each with 36 training samples and 120 test samples. Correlated knapsack uses three instances, with nested training sets of sizes $n\in\{32,64,128,256\}$ and the same 100 test samples per instance. All methods share the same training samples.

\paragraph{Baselines.}
\textsc{Ours} uses our proposed algorithm in \Cref{app:ip-learning}.
We also compare with \textsc{SVM} \citep{xavier2021learning}, Neural Diving \citep{nair2020neuraldiving}, \textsc{MLOPT} \citep{bertsimas2021online}, and Preference Learning \citep{li2025fast}. For each instance, all compression methods are calibrated to the same reduced dimension $K$ for the same set of binary variables. That is, $K$ is the dimension of the affine restriction on these binary variables.

Following the prediction approach of \citet{xavier2021learning}, \textsc{SVM} trains linear classifiers using the cost coefficients as inputs and the binary values in the optimal training solutions as labels. For UC, it also learns whether a generator's on/off status stays unchanged between consecutive periods. We use the most reliable predictions to fix variables or add equalities while matching $K$.

Neural Diving trains a graph network to predict binary variable values from the optimal training solutions, following \citet{nair2020neuraldiving}. At test time, it fixes the variables with the highest prediction confidence.

Following \citet{bertsimas2021online}, \textsc{MLOPT} trains a neural classifier to predict strategies extracted from the optimal training solutions. Each strategy records binary variable values and the constraints that hold with equality. Preference Learning uses the pairwise ranking objective of \citet{li2025fast} to compare strategies. Its scores account for objective value and constraint violations. For both methods, we fix part of the suggested assignment to match $K$.

\paragraph{Metrics.}
For test $i$, let $v_i^\star$ be the original optimum and $v_i^m$ the value returned by method $m$. These are total costs for UC and total profits for knapsack. Recovery is the percentage of tests where the returned solution is feasible for the original model and within $10^{-6}+10^{-9}|v_i^\star|$ of its optimal value. The per-test budget is 30 seconds for UC and grouped knapsack, and 20 seconds for correlated knapsack.

For each UC or correlated knapsack instance, all methods use the same anchor: an optimal solution at the first training cost. Its value at test $i$ is $v_i^a$. The relative objective ratio is
\[
 \frac{\sum_i(v_i^m-v_i^a)}{\sum_i(v_i^\star-v_i^a)}.
\]
This is the fraction of possible improvement over the anchor recovered across all tests. Higher is better, and one is optimal. For grouped knapsack, we average $v_i^m/v_i^\star$ over feasible returns.

Learning time covers the computation used to learn the compression after the training costs and corresponding optimal solutions are provided. For neural methods, this includes choosing model settings and training the network. For our method, it covers \Cref{alg:ip-cumulative} excluding the training IP solves. We exclude these solves because all methods use their solutions.

\paragraph{Implementation details.}
\textsc{Ours}, \textsc{SVM}, and \textsc{MLOPT} use one CPU thread. Neural Diving and Preference Learning use an RTX 4090 for UC and an RTX PRO 4500 for both knapsack settings.

We select model settings on UC's separate validation set. For knapsack, we hold out approximately 20\% of the training costs. After selection, each model is refitted on all $n$ training costs. \textsc{SVM} uses five-fold cross-validation, with input features standardized within each training fold.

\textsc{MLOPT} uses three ReLU hidden layers of width 32 or 64. We train it with Adam and cross-entropy loss for 50 or 150 epochs, choosing the learning rate from $\{10^{-3},10^{-2}\}$. Neural Diving uses a six-layer graph network of width 32 or 64. We train it with Adam and binary cross-entropy at learning rate $10^{-3}$, with early stopping within 150 epochs. For correlated knapsack, validation can select up to 2,000 further training updates.

Preference Learning uses two attention blocks of width 128. We train it with AdamW for 1,200 epochs and select the learning rate from $\{10^{-4},10^{-3}\}$. Its loss fits both the ordering of strategies and the differences between their quality scores.

\subsubsection{Results}
\Cref{tab:ip-matched-uc,tab:ip-matched-knapsack,tab:ip-random-knapsack} show that \textsc{Ours} achieves the highest recovery and objective ratios among the compression methods tested. Notably, our learning time is much shorter than that of the neural baselines.

\begingroup
\renewcommand{\arraystretch}{1.0}
\setlength{\intextsep}{14pt plus 2pt minus 1pt}
\setlength{\floatsep}{16pt plus 2pt minus 1pt}
\setlength{\textfloatsep}{12pt plus 2pt minus 1pt}
\begin{table}[!htbp]
\centering\small
\setlength{\tabcolsep}{5pt}
\setlength{\abovecaptionskip}{0pt}
\setlength{\belowcaptionskip}{4pt}
\caption{Results for unit commitment.}
\label{tab:ip-matched-uc}
\label{tab:ip-simple-comparison}
\begin{tabular}{rlrrrr}
\toprule
$n$ & Method & $K$ & Recovery (\%) & Relative ratio & Learning time (s) \\
\midrule
--- & \textsc{Full} & 180 & 100.00 & 1.000000 & --- \\
512 & \textbf{\textsc{Ours}} & \textbf{19} & \textbf{98.44} & \textbf{0.999952} & \textbf{0.0088} \\
 & \textsc{SVM} & 19 & 97.66 & 0.999818 & 4.18 \\
 & Neural Diving & 19 & 73.05 & 0.923578 & 86.12 \\
 & \textsc{MLOPT} & 19 & 93.75 & 0.997785 & 47.01 \\
 & Preference Learning & 19 & 43.36 & 0.763776 & 172.38 \\
\bottomrule
\end{tabular}
\end{table}

\begin{table}[!htbp]
\centering\small
\setlength{\tabcolsep}{5pt}
\setlength{\abovecaptionskip}{0pt}
\setlength{\belowcaptionskip}{4pt}
\caption{Results for grouped knapsack.}
\label{tab:ip-matched-knapsack}
\begin{tabular}{rlrrrr}
\toprule
$n$ & Method & $K$ & Recovery (\%) & Objective ratio & Learning time (s) \\
\midrule
--- & \textsc{Full} & 96 & 100.00 & 1.000000 & --- \\
36 & \textbf{\textsc{Ours}} & \textbf{32} & \textbf{100.00} & \textbf{1.000000} & \textbf{0.0072} \\
 & \textsc{SVM} & 32 & 1.39 & 0.898644 & 0.8526 \\
 & Neural Diving & 32 & 10.83 & 0.948165 & 13.17 \\
 & \textsc{MLOPT} & 32 & 9.72 & 0.953409 & 0.6562 \\
 & Preference Learning & 32 & 15.00 & 0.952625 & 122.00 \\
\bottomrule
\end{tabular}
\end{table}

\begin{table}[!htbp]
\centering\small
\setlength{\tabcolsep}{5pt}
\setlength{\abovecaptionskip}{0pt}
\setlength{\belowcaptionskip}{4pt}
\caption{Results for correlated knapsack. Each $K$ triple lists the dimensions for the three instances.}
\label{tab:ip-random-knapsack}
\begin{tabular}{rlcrrr}
\toprule
$n$ & Method & $K$ & Recovery (\%) & Relative ratio & Learning time (s) \\
\midrule
--- & \textsc{Full} & 100 & 100.00 & 1.000000 & --- \\
32 & \textbf{\textsc{Ours}} & $(29,1,31)$ & \textbf{94.00} & \textbf{0.990033} & \textbf{0.0032} \\
 & \textsc{SVM} & $(29,1,31)$ & 86.33 & 0.971049 & 0.2103 \\
 & \textsc{MLOPT} & $(29,1,31)$ & 89.33 & 0.982914 & 1.10 \\
 & Neural Diving & $(29,1,31)$ & 91.67 & 0.983863 & 103.81 \\
 & Preference Learning & $(29,1,31)$ & 86.00 & 0.972947 & 86.58 \\
\midrule
64 & \textbf{\textsc{Ours}} & $(33,1,33)$ & \textbf{96.67} & \textbf{0.995254} & \textbf{0.0072} \\
 & \textsc{SVM} & $(33,1,33)$ & 92.33 & 0.983389 & 0.2653 \\
 & \textsc{MLOPT} & $(33,1,33)$ & 96.33 & 0.994779 & 3.07 \\
 & Neural Diving & $(33,1,33)$ & 91.67 & 0.897010 & 113.66 \\
 & Preference Learning & $(33,1,33)$ & 86.33 & 0.976270 & 86.18 \\
\midrule
128 & \textbf{\textsc{Ours}} & $(35,1,37)$ & \textbf{98.67} & \textbf{0.998102} & \textbf{0.0129} \\
 & \textsc{SVM} & $(35,1,37)$ & 95.00 & 0.989084 & 0.3336 \\
 & \textsc{MLOPT} & $(35,1,37)$ & 98.33 & 0.997152 & 3.35 \\
 & Neural Diving & $(35,1,37)$ & 93.33 & 0.964404 & 124.59 \\
 & Preference Learning & $(35,1,37)$ & 89.33 & 0.973422 & 132.67 \\
\midrule
256 & \textbf{\textsc{Ours}} & $(40,1,41)$ & \textbf{99.67} & \textbf{0.999525} & \textbf{0.0244} \\
 & \textsc{SVM} & $(40,1,41)$ & 97.00 & 0.992406 & 0.5422 \\
 & \textsc{MLOPT} & $(40,1,41)$ & 99.33 & 0.999051 & 6.58 \\
 & Neural Diving & $(40,1,41)$ & 94.67 & 0.966303 & 170.96 \\
 & Preference Learning & $(40,1,41)$ & 93.00 & 0.984813 & 240.31 \\
\bottomrule
\end{tabular}
\end{table}

\FloatBarrier
\endgroup

\section{Additional Experimental Details and Results}
\label{app:beyond-prior-experiment}
This section collects the experimental material not shown in the main text.

\paragraph{LP instances.}
We write all LPs in the form
\[
    \min_{x\in\mathbb R^d}\{c^\top x:\ Ax\le b\},
\]
and keep the feasible region fixed while varying the cost vector. Equalities are written as pairs of inequalities, and variable bounds are included in $Ax\le b$. Table~\ref{tab:lp-instance-sizes} gives the resulting numbers of constraints and variables.

\paragraph{Synthetic datasets.}
We consider packing and network-flow LPs, as in \citet{sakaue2024projections}, and also include shortest-path LPs.

\textbf{Packing.} We use 360 nonnegative variables and 60 resource constraints. Each variable appears in only one of these constraints.

\textbf{MaxFlow.} We use a directed acyclic network with 64 nodes and 300 arcs. We connect pairs of parallel paths in series and add randomly selected forward arcs between internal nodes.

\textbf{MinCostFlow.} We use the same graph construction with 86 nodes and 360 arcs, including a direct source-to-sink arc. All arcs have unit capacity.

\textbf{ShortestPath.} We use a $16\times16$ directed grid with rightward and downward arcs.

We also include four random bounded LPs, denoted by RandomLP A--D, with Gaussian constraint matrices.

\paragraph{Netlib datasets.}
We use GROW7, SC205, SCAGR25, and STAIR from Netlib. These instances are preprocessed similarly to \citet{sakaue2024projections}, and their reported dimensions refer to the resulting inequality-form LPs.

\begin{table}[H]
\centering
\small
\caption{Here \(m\) is the number of inequality constraints and \(d\) is the number of variables. Lower and upper bounds on variables are included in \(m\). Each equality is replaced by two inequalities.}
\label{tab:lp-instance-sizes}
\begingroup
\begin{tabular*}{\textwidth}{@{\extracolsep{\fill}}lrrrr@{}}
\toprule
Instance & Packing & MaxFlow & MinCostFlow & ShortestPath \\
\midrule
$m$ & 420 & 724 & 892 & 1472 \\
$d$ & 360 & 300 & 360 & 480 \\
\midrule
Instance & RandomLP A & RandomLP B & RandomLP C & RandomLP D \\
\midrule
$m$ & 380 & 500 & 620 & 740 \\
$d$ & 140 & 180 & 220 & 260 \\
\midrule
Instance & GROW7 & SC205 & SCAGR25 & STAIR \\
\midrule
$m$ & 861 & 499 & 1271 & 1126 \\
$d$ & 301 & 203 & 500 & 473 \\
\bottomrule
\end{tabular*}
\endgroup
\end{table}

\paragraph{Cost perturbations.}
We use shared factors to vary the relative rewards of competing activities or the relative costs of alternative routes.
For additive Gaussian cost perturbations, we use the factor model
\[
    c = c_0 + U_c \theta,\qquad
    \theta_j = \sigma_j z_j,\qquad z_j\overset{\mathrm{i.i.d.}}{\sim}\mathcal N(0,1),
\]
where $c_0\in\mathbb R^d$ is the nominal cost vector and
$U_c\in\mathbb R^{d\times r_c}$ contains the cost directions.
Before clipping, the factors have standard deviations $\sigma_j$ that are equal or decrease by a fixed factor across directions. The released code specifies the sampling parameters.
The exact learner in the known-prior experiment is given the
ball $\mathcal C_i=\{c:\|c-c_{0,i}\|_2\le R_i^{\mathcal C}\}$.
We radially clip each Gaussian perturbation to ensure $\|c-c_0\|_2\le R_i$.

In the unknown-prior experiments the learner is not given
$\mathcal C_i$.  The costs are generated in the same way but are not clipped:
\[
    c=c_0+U_c\theta,\qquad
    \theta_j\sim\mathcal N\!\left(0,(\eta_i\sigma_j)^2\right),
\]
where $\eta_i$ is an experiment-specific scale multiplier.

\begin{table}[H]
\centering
\small
\setlength{\tabcolsep}{5pt}
\renewcommand{\arraystretch}{1.15}
\caption{Default LP sampling parameters.}
\label{tab:lp-sampling-values}
\begingroup
\resizebox{\textwidth}{!}{%
\begin{tabular}{lrlrrr}
\toprule
Instance & $r_c$ & $\sigma_j$ & $R_i^{\mathcal C}$ & $R_i$ & $\eta_i$ \\
\midrule
Packing & 28 & $0.92R_i(0.94)^{j-1}$ & $1.694$ & $1.524$ & $0.35$ \\
MaxFlow & 12 & $0.95R_i$ & $1.278\times10^{-3}$ & $1.150\times10^{-3}$ & --- \\
MinCostFlow & 28 & $0.72R_i(0.94)^{j-1}$ & $3.968$ & $3.492$ & $0.35$ \\
ShortestPath & 16 & $0.70R_i(0.82)^{j-1}$ & $0.04725$ & $0.04488$ & --- \\
RandomLP A & 24 & $0.70R_i(0.82)^{j-1}$ & $2.866\times10^{-3}$ & $2.637\times10^{-3}$ & $20$ \\
RandomLP B & 26 & $0.70R_i(0.82)^{j-1}$ & $8.069\times10^{-3}$ & $7.262\times10^{-3}$ & $20$ \\
RandomLP C & 28 & $0.70R_i(0.82)^{j-1}$ & $2.145\times10^{-3}$ & $1.888\times10^{-3}$ & $20$ \\
RandomLP D & 30 & $0.70R_i(0.82)^{j-1}$ & $2.398\times10^{-3}$ & $2.063\times10^{-3}$ & $20$ \\
GROW7 & 20 & $0.70R_i(0.82)^{j-1}$ & $0.01263$ & $0.01162$ & $0.35$ \\
SC205 & 20 & $0.70R_i(0.82)^{j-1}$ & $3.891\times10^{-4}$ & $3.580\times10^{-4}$ & $0.35$ \\
SCAGR25 & 24 & $0.70R_i(0.82)^{j-1}$ & $0.04802$ & $0.04418$ & --- \\
STAIR & 10 & $0.70R_i(0.82)^{j-1}$ & $8.144\times10^{-5}$ & $1.629\times10^{-6}$ & --- \\
\bottomrule
\end{tabular}%
}
\endgroup
\par\smallskip
\begin{minipage}{\textwidth}
\raggedright
Dashes in the $\eta_i$ column mean not applicable (we did not run unknown-prior experiments for these instances).
\end{minipage}
\end{table}

\FloatBarrier
\begingroup{}
\subsection{Dimension growth under an estimated prior}\label{app:estimated-rank-growth}
\begin{figure}[H]
\begingroup{}
\centering
\includegraphics[width=0.97\textwidth]{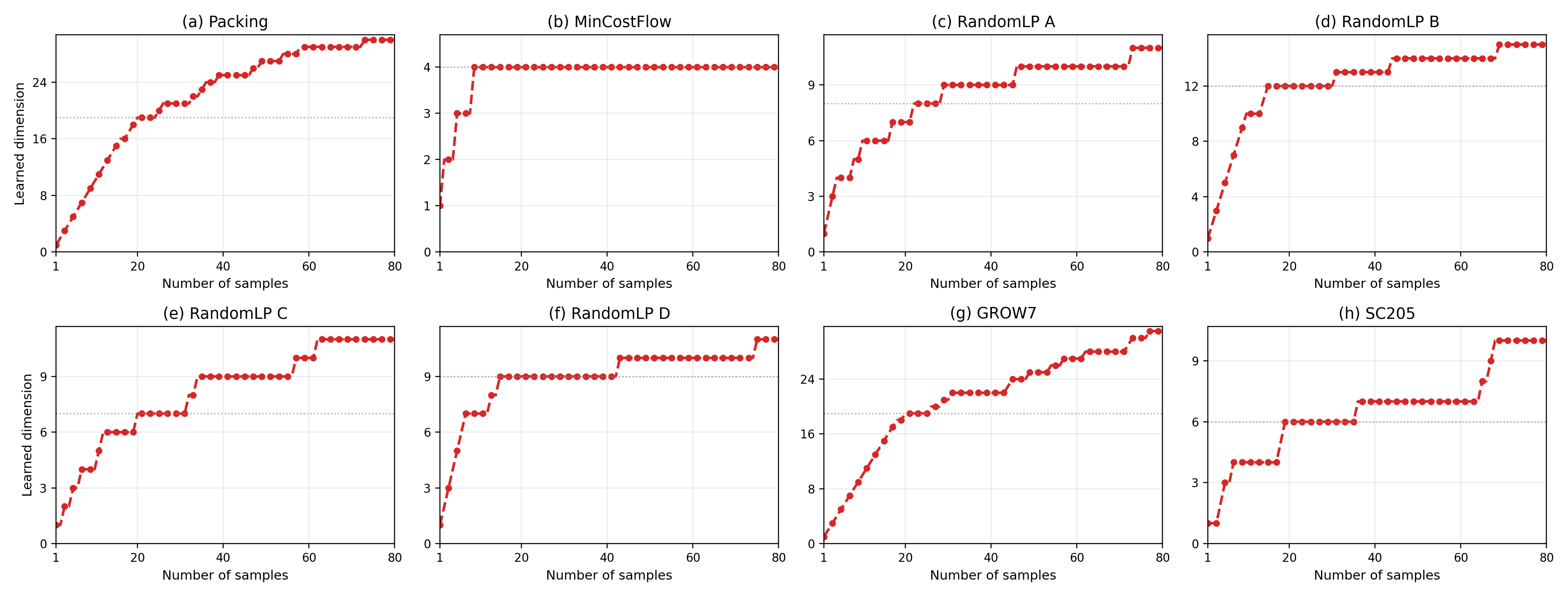}
\caption{Learned dimension of \textsc{Ours}. LP sizes are given in Table~\ref{tab:lp-instance-sizes}.}
\label{fig:beyond-prior-rank-growth}
\endgroup{}
\end{figure}

\Cref{fig:beyond-prior-rank-growth} reports the dimension learned from the costs retained by the estimated prior. The MinCostFlow curve plateaus early. Packing and GROW7 attain larger learned dimensions than the other instances. These dimensions reflect the optimizer differences observed in the retained sample.
\endgroup{}

\subsection{Known-prior benchmark}
\label{app:natural-prior-experiments}

\paragraph{Protocol.}
\begingroup{}The known-prior benchmark uses Packing, MaxFlow, MinCostFlow, ShortestPath, RandomLP A--D, and the Netlib instances GROW7, SC205, SCAGR25, and STAIR.\endgroup{}  In this experiment the exact learner is given the intended prior set and is denoted \textsc{OursExact}.  We compare it with \textsc{Rand}, \textsc{DataDrivenProj}, \textsc{FCNN}, and \textsc{Full}, as discussed in \cref{sec:experiments}.

In the reduced-dimension sweep, the projection baselines are evaluated at $K\in\{5,10,20,30,40,50\}$.  \textsc{OursExact} learns its dimension from the training costs. Its objective ratio is therefore shown as a reference across all displayed $K$ values.  In the sample-efficiency study, $K=20$ is fixed only for the approximate baselines, while \textsc{OursExact} again uses the dimension discovered by \Cref{alg:adj-cumulative}.  The main metric is the average test objective ratio.

Standard errors are estimated by dividing the sample standard deviation by the square root of the number of observations.
\begin{figure}[H]
\centering
\includegraphics[width=0.97\textwidth]{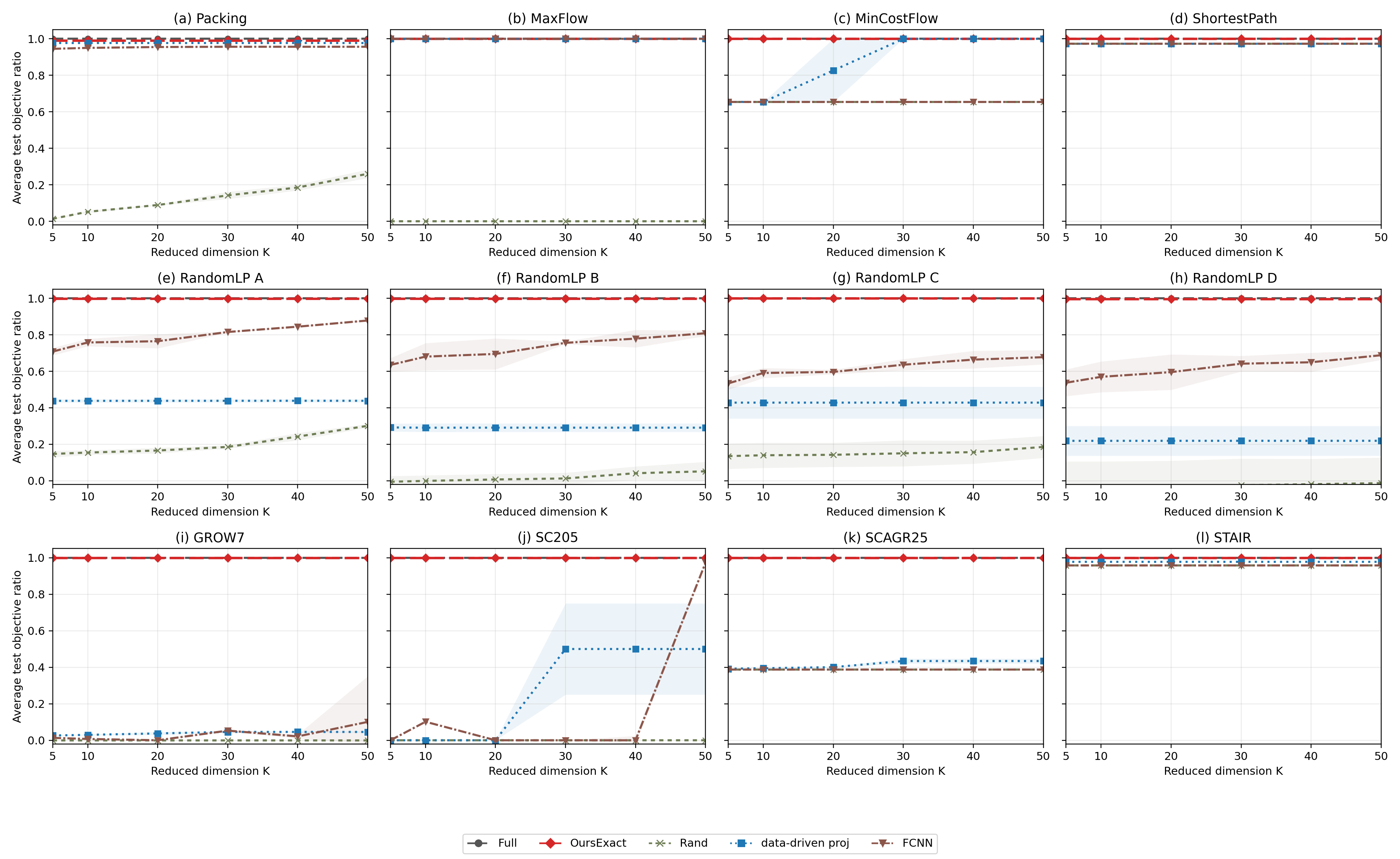}
\caption{Objective ratio as the projection-baseline dimension $K$ varies under known priors.  Our method is plotted as a reference across all displayed $K$ values.}
\label{fig:k-sweep-fixedC}
\end{figure}

\Cref{fig:k-sweep-fixedC} shows higher objective ratios for \textsc{OursExact} on the RandomLP families and several Netlib instances.  \textsc{OursExact} attains objective ratios close to one across these panels, whereas \textsc{DataDrivenProj} remains far below one on RandomLP B/D, GROW7, SCAGR25, and much of SC205.  \textsc{FCNN} is often the strongest approximate baseline on RandomLP A--D and can improve at larger $K$, but it still falls short of exact compression except in a few large-$K$ cases.

\begingroup{}On Packing, MaxFlow, ShortestPath, and STAIR, the learned projection baselines attain objective ratios close to one. On MinCostFlow, \textsc{DataDrivenProj} approaches the full-LP objective as $K$ increases. The \textsc{OursExact} reference uses the same learned compression throughout the sweep.\endgroup{}

\begin{figure}[H]
\centering
\includegraphics[width=0.97\textwidth]{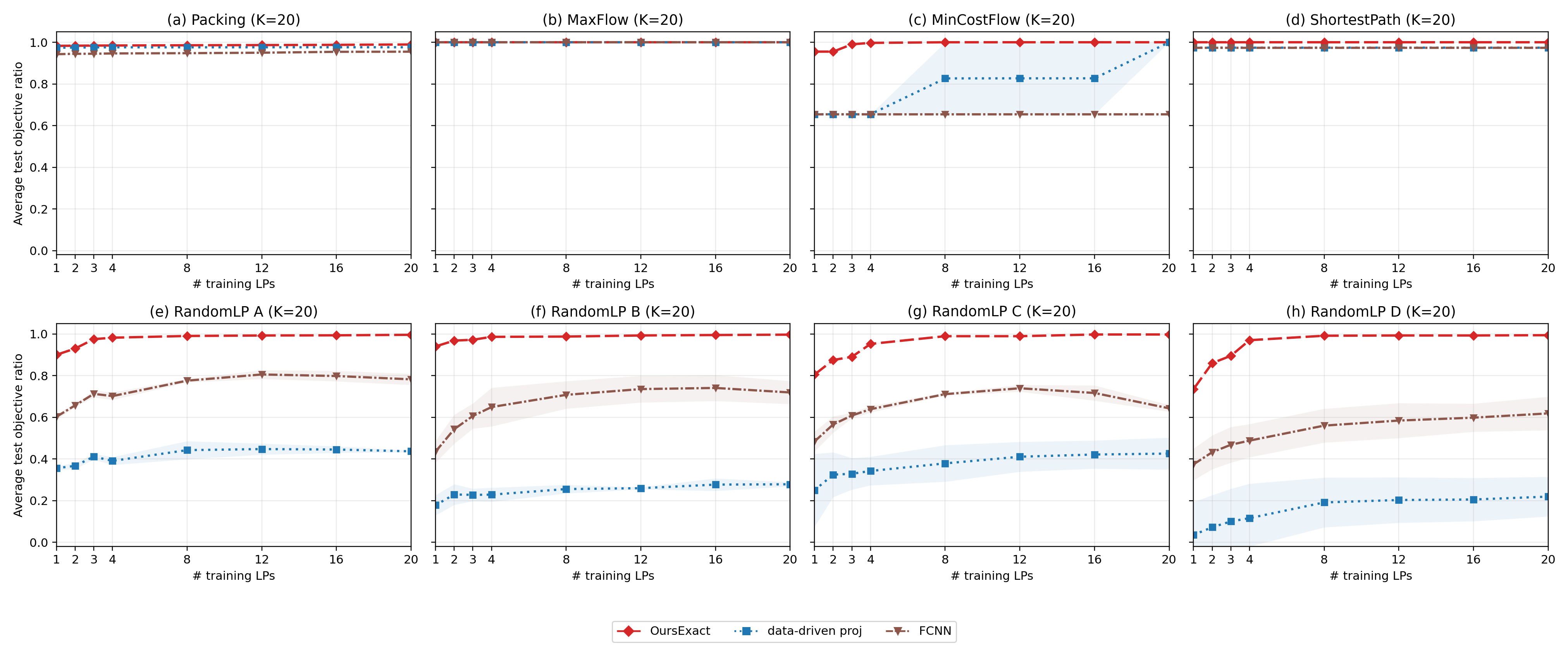}
\caption{Objective ratio as the sample size varies under known priors.}
\label{fig:sample-efficiency}
\end{figure}

\begin{figure}[H]
\centering
\includegraphics[width=0.97\textwidth]{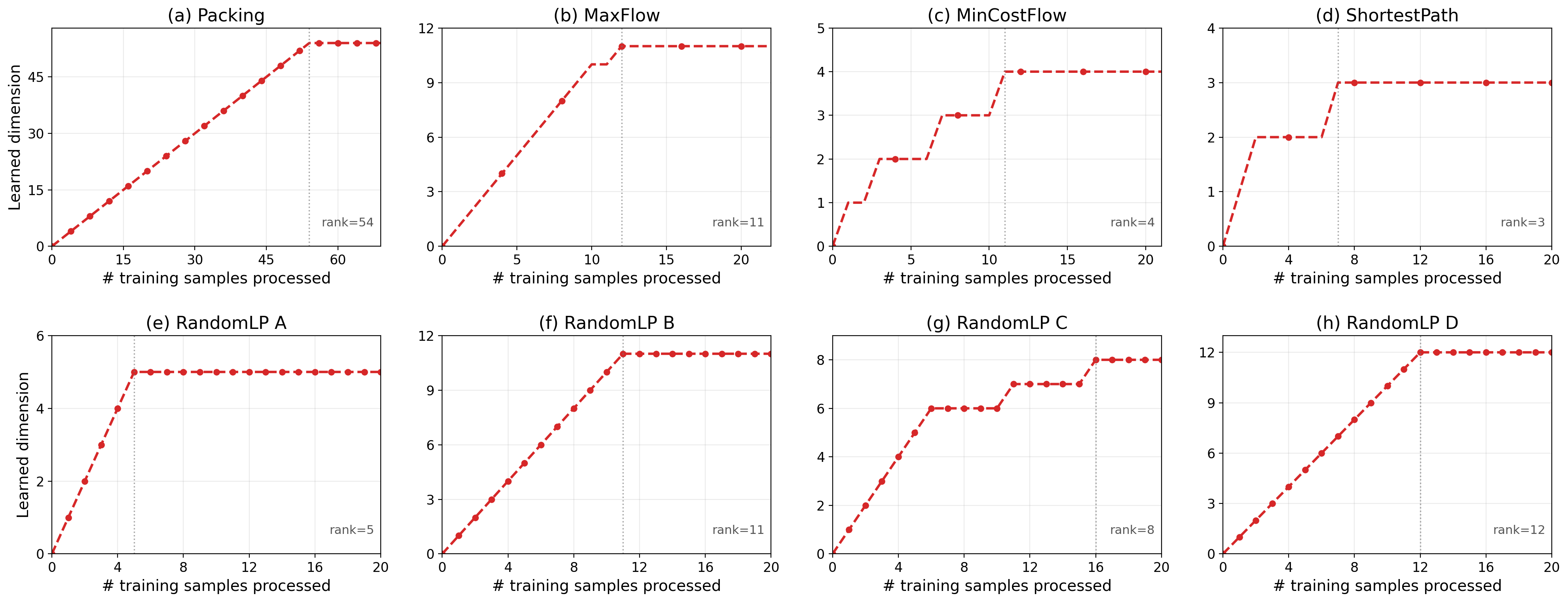}
\caption{Learned dimension growth of our algorithm on the original synthetic LP instances.}
\label{fig:rank-growth}
\end{figure}

\Cref{fig:sample-efficiency} highlights the sample-efficiency difference in the known-prior setting.  On RandomLP A--D, \textsc{OursExact} quickly reaches \begingroup{}an objective ratio close to one\endgroup{} as more optimal-face directions are observed, while \textsc{DataDrivenProj} stays around a substantially lower plateau and \textsc{FCNN} improves more gradually.  On Packing, MaxFlow, and ShortestPath, both \textsc{OursExact} and the projection baselines attain objective ratios close to one.

\Cref{fig:rank-growth} reports the learned dimensions on the same synthetic LP instances.  This figure is the known-prior analogue of \Cref{fig:beyond-prior-rank-growth}.

\FloatBarrier
\subsection{Additional unknown-prior diagnostics}
\label{app:additional-beyond-prior-diagnostics}

\paragraph{Protocol.}
\begingroup{}
These additional experiments use the ridge Mahalanobis score with a threshold equal to the largest training score. The resulting estimated prior contains all observed training costs. The same costs are used to estimate the prior and learn the compression, so the displayed sample budgets count the historical costs available to each method.
\endgroup{}
The projection baselines are trained directly on the observed cost samples and do not construct an estimated prior.

\begin{figure}[H]
\centering
\includegraphics[width=0.97\textwidth]{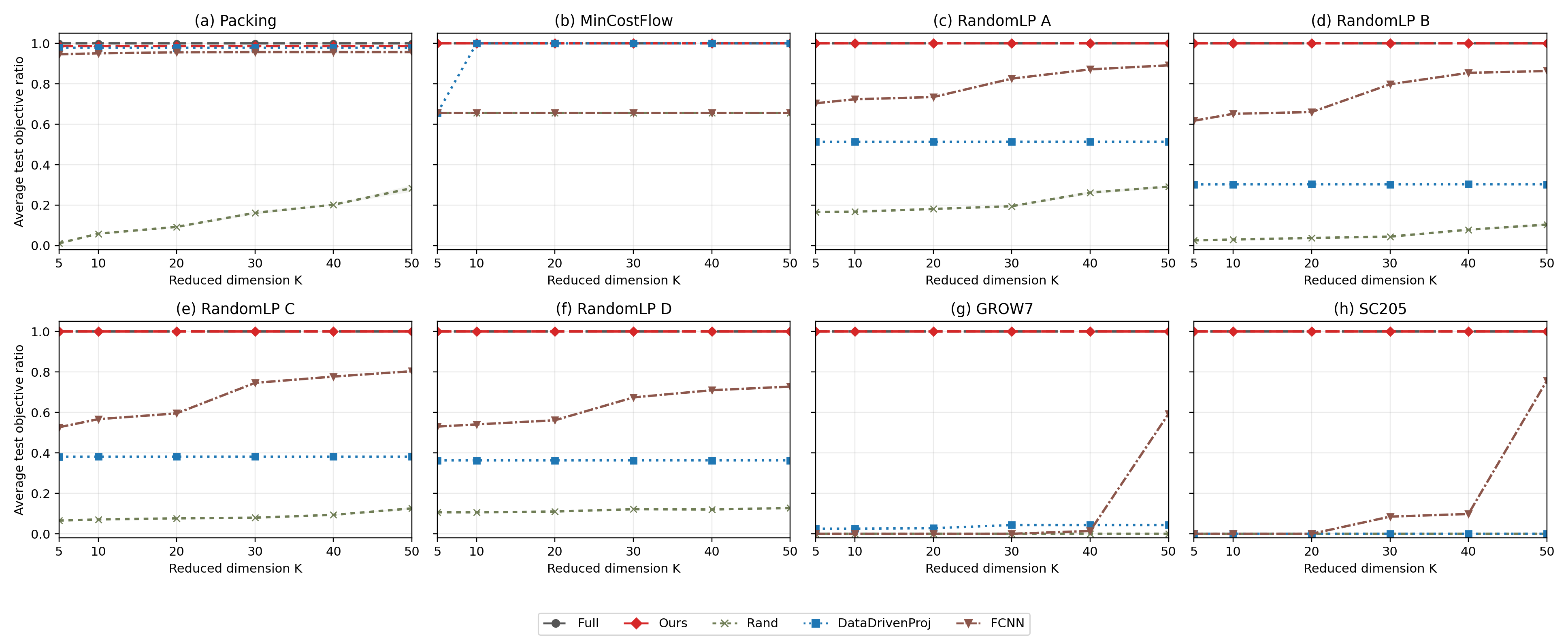}
\caption{Reduced-dimension sweep with an unknown prior set.}
\label{fig:beyond-prior-k-sweep}
\end{figure}

\Cref{fig:beyond-prior-k-sweep} \begingroup{}compares objective ratios as the baseline projection dimension varies.\endgroup{}  \begingroup{}\textsc{Ours} attains objective ratios close to one.\endgroup{}  The baselines can be \begingroup{}strong on Packing and MinCostFlow\endgroup{}, but they remain substantially below exact compression on RandomLP B--D and GROW7.

\begin{figure}[H]
\centering
\includegraphics[width=0.97\textwidth]{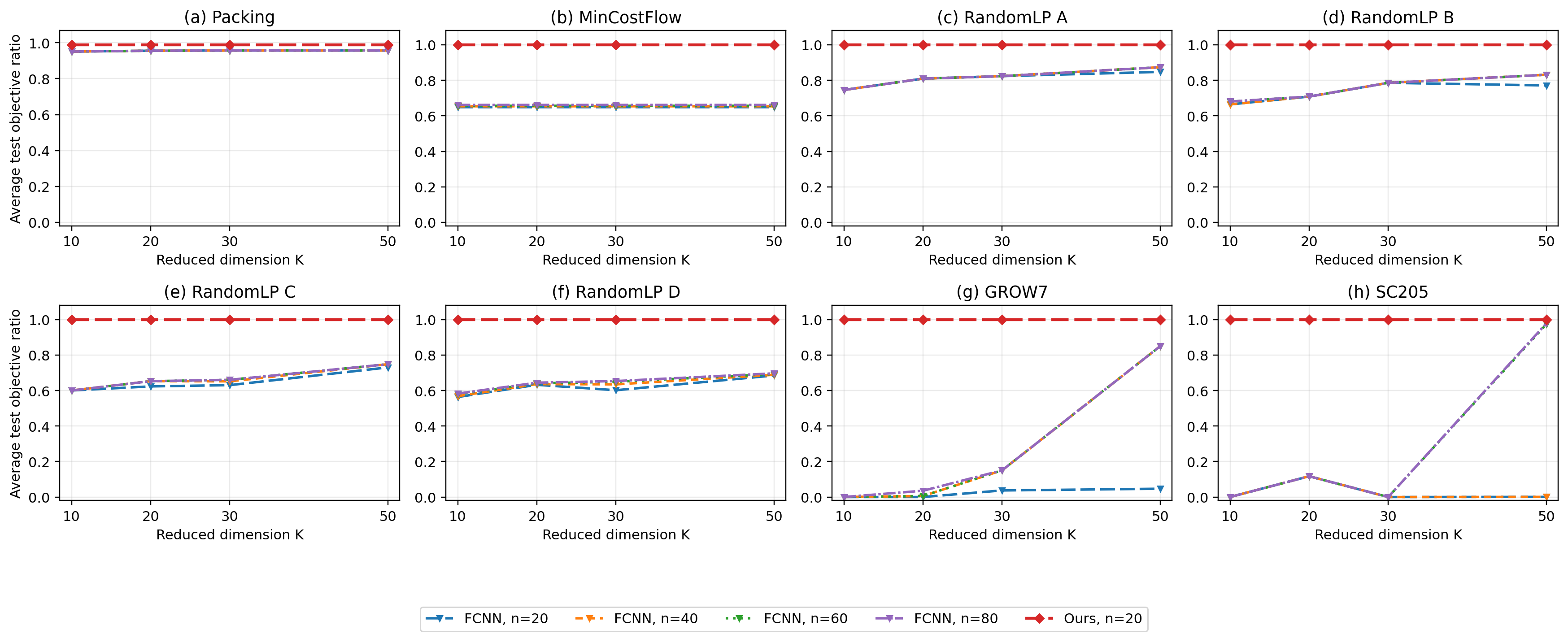}
\caption{\begingroup{}Reduced-dimension sweep for \textsc{FCNN}, with the $n=20$ \textsc{Ours} result as a reference.\endgroup{}}
\label{fig:beyond-prior-costonly-k-sweep}
\end{figure}

\Cref{fig:beyond-prior-costonly-k-sweep} complements \Cref{fig:beyond-prior-costonly-samples} \begingroup{}by varying the reduced dimension $K$ for \textsc{FCNN} at each displayed training-sample size.\endgroup{} \Cref{fig:beyond-prior-datadriven-samples,fig:beyond-prior-datadriven-k-sweep} report the corresponding diagnostics for the \textsc{DataDrivenProj} baseline.

\begin{figure}[H]
\centering
\includegraphics[width=0.97\textwidth]{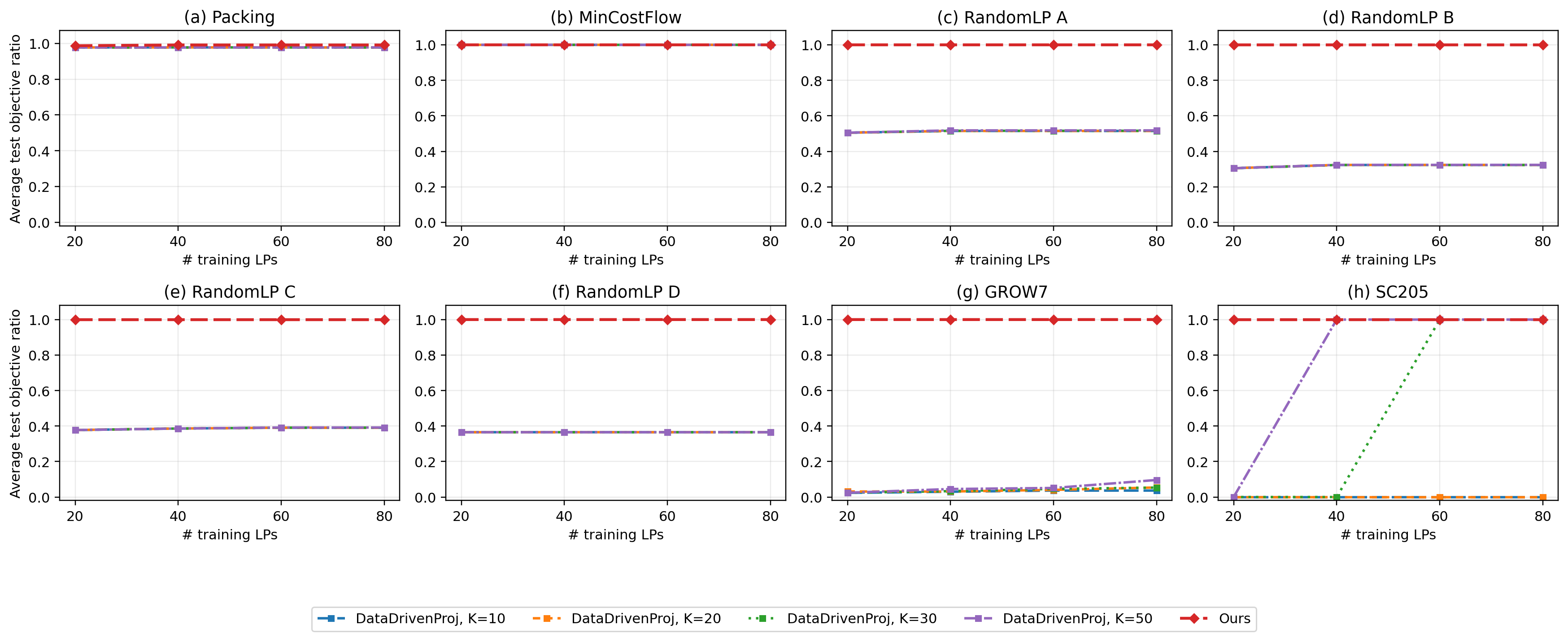}
\caption{Sample-efficiency study for the \textsc{DataDrivenProj} baseline, with \textsc{Ours} shown as the reference.}
\label{fig:beyond-prior-datadriven-samples}
\end{figure}

\begin{figure}[H]
\centering
\includegraphics[width=0.97\textwidth]{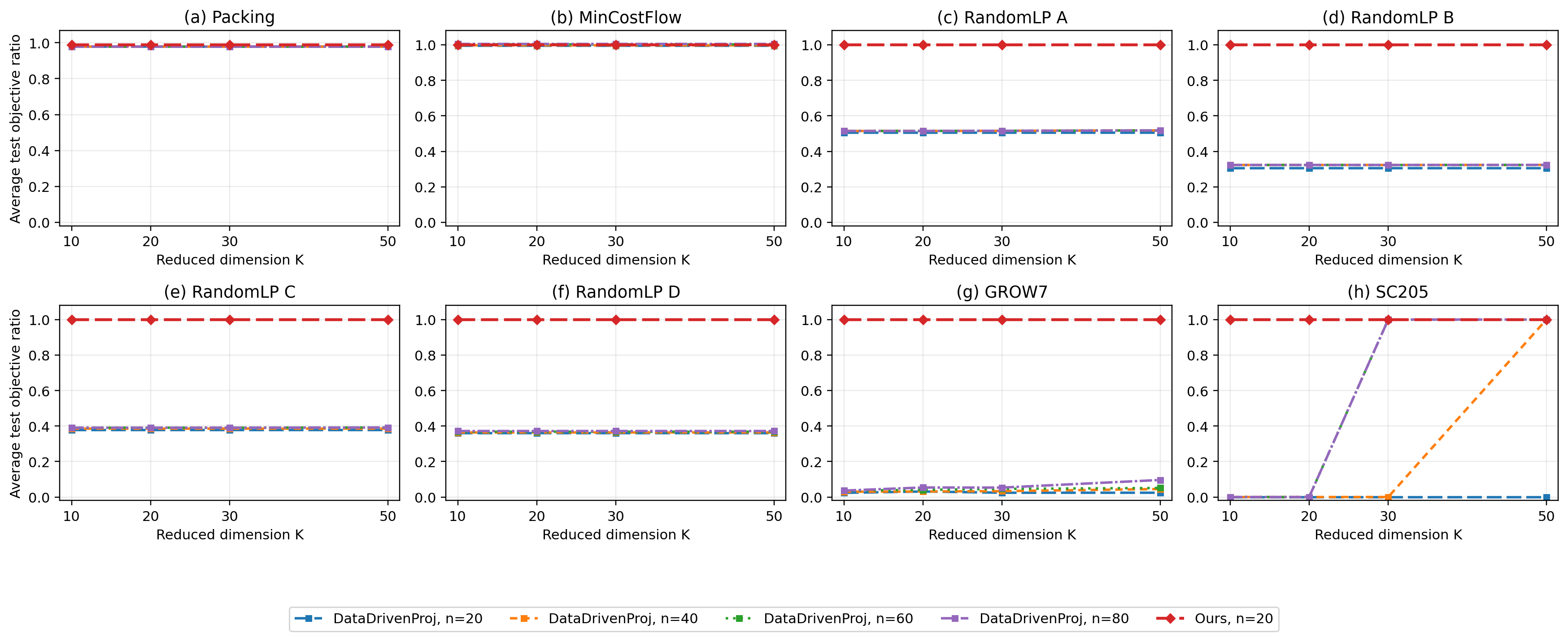}
\caption{Reduced-dimension sweep for the \textsc{DataDrivenProj} baseline, with the $n=20$ \textsc{Ours} result shown as the reference.}
\label{fig:beyond-prior-datadriven-k-sweep}
\end{figure}

\end{document}